\documentclass[english]{amsart}

\usepackage[all,cmtip]{xy}
\usepackage[T1]{fontenc}
\usepackage{url}
\usepackage[frenchb, main=english]{babel}
\usepackage[utf8]{inputenc}
\usepackage{amsmath,amssymb,amscd,amsfonts}
\usepackage{enumerate}
\usepackage{enumitem}
\usepackage{babel}
\usepackage{amstext}
\usepackage{amsmath}
\usepackage{bbold}
\usepackage{amsfonts}
\usepackage{latexsym}
\usepackage{amsfonts}
\usepackage{ifthen}
\usepackage{xypic}
\xyoption{all}
\usepackage{color}

\newcommand{\R}{\mathbb{R}}
\newcommand{\CC}{\mathbb{C}}
\newcommand{\Q}{\mathbb{Q}}

\newcommand{\Z}{\mathbb{Z}}

\newcommand{\sE}{\mathcal{E}}
\newcommand{\sA}{\mathcal{A}}

\newcommand{\ddbar}{\partial\bar{\partial}}

\newcommand{\wh}{\widehat}
\newcommand{\wt}{\widetilde}

\newcommand{\cA}{\mathcal{A}}
\newcommand{\cF}{\mathcal{F}}

\newcommand{\cC}{\mathcal{C}}

\newcommand{\cG}{\mathcal{G}}

\renewcommand{\O}{\mathcal{O}}

\newcommand{\ep}{\varepsilon}
\renewcommand{\epsilon}{\varepsilon}

\newcommand{\Id}{\mathrm{Id}}

\renewcommand{\ker}{\mathrm{Ker} \,}

\newcommand{\ol}{\overline}

\renewcommand{\leq}{\leqslant}
\renewcommand{\geq}{\geqslant}

\newcommand{\Jac}{\mathrm{Jac}}
\newcommand{\codim}{\mathrm{codim}}
\newcommand{\Tr}{\mathrm{Tr}}

\newcommand{\ddc}{dd^c}

\newcommand{\Supp}{\mathrm {Supp}}

\newcommand{\Vol}{\mathrm{Vol}}

\newcommand{\End}{\mathrm{End}}

\newcommand{\Hom}{\mathrm{Hom}}

\newcommand{\rk}{\mathrm{rk}}

\newcommand{\Tor}{\mathrm{Tor}}

\newcommand{\dbar}{\bar \partial}

\newtheorem{thm}{Theorem}[section]
\newtheorem{lemme}[thm]{Lemma}
\newtheorem{proposition}[thm]{Proposition}

\newtheorem{cor}[thm]{Corollary}
\theoremstyle{remark}

\newtheorem{remark}[thm]{Remark}

\numberwithin{equation}{thm}

\title{Hermite--Einstein metrics in singular settings}

\author[Cao]{Junyan Cao}
\address{Laboratoire de Mathématiques J.A. Dieudonné UMR 7351 CNRS, Université Côte d'Azur Parc Valrose 06108, NICE CEDEX 2, France}
\email{junyan.cao@unice.fr}

\author[Graf]{Patrick Graf}
\address{Lehrstuhl f\"ur Mathematik I, Universit\"at Bayreuth, 95440 Bayreuth, Germany}
\email{patrick.graf@uni-bayreuth.de}

\author[Naumann]{Philipp Naumann}
\address{Lehrstuhl f\"ur Mathematik VIII, Universit\"at Bayreuth, 95440 Bayreuth, Germany}
\email{philipp.naumann@uni-bayreuth.de}

\author[Paun]{Mihai Paun}
\address{Lehrstuhl f\"ur Mathematik VIII, Universit\"at Bayreuth, 95440 Bayreuth, Germany}
\email{mihai.paun@uni-bayreuth.de}

\author[Peternell]{Thomas Peternell}
\address{Lehrstuhl f\"ur Mathematik I, Universit\"at Bayreuth, 95440 Bayreuth, Germany}
\email{thomas.peternell@uni-bayreuth.de}

\author[Wu]{Xiaojun Wu}
\address{Lehrstuhl f\"ur Mathematik VIII, Universit\"at Bayreuth, 95440 Bayreuth, Germany}
\email{xiaojun.wu@uni-bayreuth.de}

\begin{document}

\maketitle

\section{Preliminaries and main Results}\label{intro}

\subsection{Abstract}

Let $(X, \omega_X)$ be a normal Kähler space in the sense of Grauert, i.e.~the Kähler metric $\omega_X$ has local potentials which are induced by smooth, strictly plurisubharmonic functions defined on the (local) ambient space.
In this article we pursue the following main goals.
\smallskip

 \noindent $\bullet$ We establish the existence of Hermite--Einstein metrics for $\omega_X$-stable reflexive coherent sheaves $\cF$ on $X$.
Our method relies on a recent, very important work of Guo, Phong and Sturm~\cite{Phong}.
It is explained in Sections~1--5.
\smallskip

 \noindent $\bullet$ Assume that the space $X$ has at most klt singularities.
Then it has quotient singularities in the complement of a set $Z\subset X$ of codimension at least three, by the work of~\cite{GKKP11}.
If moreover $\cF|_{X \setminus Z}$ admits a $\Q$-vector bundle structure, then we obtain a more precise result.
Namely, we show that given any open subset $U \subset X$ containing $Z$, the said Hermite--Einstein metric on $\cF$ is compatible with the quotient structure when restricted to $X \setminus U$.
These results are established in Section~6 and 7.
\smallskip

 \noindent $\bullet$ The main motivation for the questions treated in this article was a conjecture formulated by Campana, Höring and Peternell in~\cite[Conj.~0.1]{CHP22}.
We discuss a particular case of it in Section~8.
\medskip

\subsection{Acknowledgements} It is our great pleasure to thank D.H. Phong and his collaborators for very enjoyable, interesting 
e-exchanges and ample explanations concerning their work. Many thanks to H. Guenancia and S. Sun for their pertinent remarks about earlier versions of the current article.

 \medskip


\subsection{Hermitian metrics on reflexive sheaves} 
 
 \noindent In order to state our main results, we start by recalling the notion of metric for a coherent, torsion free sheaf in the spirit of 
 Grauert and Riemenschneider in \cite{GR70}.
\medskip

\noindent Let $(X, \omega_X)$ be a compact normal K\"ahler space. We fix a finite covering $(A_\alpha)_\alpha$ of $X$ such that for each index $\alpha$ the corresponding analytic set $A_\alpha$ admits an embedding in the unit ball $U_\alpha$ in $\CC^{N_\alpha}$. Also, here we adopt the definition according to which the metric $\displaystyle \omega_X|_{A_\alpha}$ is equal to the $dd^c$ of a smooth, strictly psh function defined on $U_\alpha$ restricted to $X$.
\smallskip

\noindent We denote by $\cF$ a reflexive sheaf on $X$; we can assume that we have the exact sequence
\begin{equation}\label{mt1}
 {\mathcal O}_{A_\alpha}^{p_\alpha}\to {\mathcal O}_{A_\alpha}^{q_\alpha}\to \cF|_{A_\alpha}\to 0
\end{equation}
for each index $\alpha$, where $p_\alpha$ and $q_\alpha$ are positive integers. We are using these maps in order to construct a Hermitian metric $h_{\cF, 0}$ on $\cF$ as follows. 
\smallskip

\noindent Let $(\rho_\alpha)_\alpha$ be a partition of unit subordinate to 
$(A_\alpha)_\alpha$ and let $s_1, s_2$ be two local sections of $\cF$; we define their scalar product $h_{\cF, 0}$ by the formula 
\begin{equation}\label{mt2}
\langle s_1, s_2\rangle := \sum \rho_\alpha \langle \sigma_{\alpha 1},  \sigma_{\alpha 2}\rangle_\alpha
\end{equation}
where $\displaystyle \langle \cdot, \cdot \rangle_\alpha$ is the -trivial- scalar product on 
${\mathcal O}_{A_\alpha}^{q_\alpha}$ and $\sigma_i$ are sections of ${\mathcal O}_{A_\alpha}^{q_\alpha}$ such that 
$$\pi_\alpha(\sigma_{\alpha i})= s_i,\qquad \langle \sigma_{\alpha i}, \tau\rangle_\alpha= 0$$
for any section $\tau$ in the image of the first map of the sequence \eqref{mt1}.
We denote by $\pi_\alpha$ the second application in \eqref{mt1}. 
\smallskip

\noindent For example, consider the maximum ideal sheaf
$\displaystyle \cF:= (z_1, z_2)\subset \mathcal O_{\mathbb C^2}$, together with the maps
\[\alpha(f):= (z_2f, -z_1f), \qquad \beta(f_1, f_2):= z_1f_1+ z_2f_2.\] 
They induce an exact sequence
$\displaystyle 0\to \mathcal  O\to \mathcal  O^2\to \cF\to 0$, from which one can compute the metric. If $s:= z_1$, then we get
\[ |s|^2_{h_\cF}= \frac{|z_1|^2}{|z_1|^2+ |z_2|^2}.\]
\smallskip

\noindent Another indication showing that the metric $h_{\cF, 0}$ is natural is given by its connection with the desingularisation of the sheaf 
$\cF$, as follows. An important result due to H. Rossi, cf. \cite{Ros68}, shows the existence of a birational map $\pi:\wt X\to X$ such that 
$\wt X$ is non-singular and such that the inverse image $E:= \pi^\star(\cF)/\Tor$ of our sheaf $\cF$ modulo its torsion is 
locally free. 

\noindent Let $Z$ be the union of the singular set of $X$ and $\cF$, and set $X_0 : =X \setminus Z$.
In sub-section \ref{Ross} we prove the following result (see Lemma \ref{lowerbound} as well as the remarks immediately afterwards).

\begin{thm}\label{RoS} The bundle $E$ admits a Hermitian metric $h_E$ such that there exists a constant $C> 0$ for which the inequality
\[\sqrt{-1}\Theta(E, h_E)\geq -C\pi^\star(\omega_X)\otimes \Id_E\]
is satisfied. Moreover, the restriction of the metric $h_E$ to the $\pi$-inverse image of $X_0$ is induced by a metric $h_{\cF}$ on $\cF$ constructed by a similar procedure as above.
\end{thm}

\noindent More precisely, the metric $h_\cF$ will be the dual of the metric $h_{\cF^\star, 0}$ defined on $\cF^\star$ as in \eqref{mt2}.

\medskip

\noindent  On the vector bundle $\displaystyle \cF|_{X_0}$ the metric  $h_{\cF, 0}$ is obtained by piecing together the quotient metrics induced by \eqref{mt1}.
We denote by $\phi_Z\leq 0$ a function with log-poles along the generators of the ideal of $Z$ and then we have the following statement.

\begin{thm}\label{HE} Let $\cF$ be a reflexive coherent sheaf on a normal Kähler space $(X, \omega_X)$. We assume that $\cF$ is stable with respect to the Kähler metric $\omega_X$. Then we can construct a metric $h_{\cF}= h_{\cF, 0}\exp(s)$ on $\cF|_{X_0}$ with the following properties:
\begin{enumerate}
\smallskip

\item[\rm (i)] It is smooth and it satisfies the Hermite-Einstein equation on $X_0$.
\smallskip

\item[\rm (ii)] There exists a constant $C>0$ such that the inequalities
$$\Tr \exp(s) \leq C,\qquad \exp(s)\geq  Ce^{N\phi_Z}\Id_{\cF}$$
hold pointwise on $X_0$, where $N> 0$ is a positive, large enough constant.
\smallskip

\item[\rm (iii)] We have $$\int_{X_0}
\frac{1}{-\phi_Z}|D's|^2dV\leq C, \quad \int_{X_0} |\Theta(\cF, h_{\cF})|^2dV< \infty$$
where the volume element $dV$ is induced by the metric $\omega_X$ and the norm $|\cdot|$ is with respect to $h_{\cF, 0}$.
\end{enumerate}
\end{thm}

\noindent Thus, our main contributions to the HE theme can be summarised as follows.
\begin{enumerate}

\item \emph{We allow both $X$ and $\cF$ to be singular.} Modulo the work of \cite{Phong}, our proof of Theorem
  \ref{HE} is  self-contained, in the sense that it only relies on the basic arguments of S. Donaldson \cite{Don85}, Uhlenbeck-Yau \cite{UY86} and C. Simpson \cite{Sim88}. The literature on this subject is abundant, and we only refer to the most recent preprints \cite{CW}, \cite{Chen23}, \cite{CT23} and the references therein for the construction of admissible Hermite-Einstein metrics, as well as for new and very promising points of view in this circle of ideas in \cite{Jon}.
\smallskip
  
\item \emph{We obtain estimates.} Indeed, the estimates in (ii), (iii) above are more precise than producing an admissible HE metric, cf. \cite{BS94} and the references quoted above.   
\end{enumerate}

\noindent In some sense, the proof of Theorem \ref{HE} is quite standard, and it will be done along the following lines. By using the aforementioned result of Rossi combined with 
a theorem due to M.~Toma, cf. \cite{Tom19}, the bundle $E$ is stable with respect to any positive, small enough perturbation $\omega_t$ of the inverse image metric $\pi^\star \omega_X$. We can therefore construct a Hermitian metric on $E$, which differs from $h_E$ obtained in 
Theorem \ref{RoS} above by a positive endomorphism of $E$ and which is Hermite-Einstein with respect to $\omega_t$. The tricky part is to extract a limit as $t\to 0$. It is at this point that the mean-value inequality in \cite{Phong} is playing a crucial role: thanks to it we obtain uniform estimates for our family of endomorphisms, which combine nicely with the approach of Simpson in \cite{Sim88}, and ultimately allow us to conclude.

\medskip  

\noindent Assume next that locally on some open subset 
$\sA\subset X$, the singularities of $X$ are quotient and the restriction $\cF|_A$ is a $\Q$-sheaf.
As consequence of Theorem \ref{HE}
we show that the endomorphism $s$ defining the HE metric $h_{\cF}|_A$ is bounded when restricted to $\sA$, cf. Theorem \ref{orbi, II}. Thus, the metric $h_{\cF}|_A$ that we construct is $\mathcal C^0$-close to the orbifold metric of $\cF$. Since we are unable to obtain higher order estimates,
this is not precise enough in order to allow us to evaluate e.g. the orbifold Chern classes of $\cF$.
\medskip  

\noindent This issue is partly addressed in our next result. Let $X$ be a compact K\"ahler space with at most klt singularities. Then by the results in \cite{GKKP11}
there exists an analytic subset $W\subset X$ such that $\codim_X W\geq 3$ and such that $X\setminus W$ has at most quotient singularities. 
\smallskip

\noindent It would be ideal to have at our disposal a Kähler metric on $X$ whose restriction to 
$X\setminus W$ accounts for the quotient singularities, i.e. its pull-back via the local uniformisations is 
quasi-isometric to the flat metric.
Unfortunately, we see no natural candidate for such a metric.  
Our substitute for it is contained in the following statement.
\begin{lemme}\label{omorb} Let $U\subset X$ be any open subset of $X$ such that $W\subset U$.  Let $U'$ be some open set such that $W \subset U' \Subset U$. 
	Then there exists a K\"ahler metric $\omega_{\rm orb}$
  whose restriction to $U'$ is smooth, and whose restriction to $X\setminus U$ has conic singularities, in the sense that its pull-back via the uniformising morphisms is quasi-isometric with the flat metric.
\end{lemme}  
\medskip

\noindent If $X$ is projective, the existence of $\omega_{\rm orb}$ is a consequence of the following ``orbifold partial resolution'' theorem, a very important result due to C. Li and G. Tian, cf.~\cite[Thm.~3]{LT19}:
\emph{there exists a projective variety $\wh X$ with at most quotient singularities, together with a birational map $p \colon \wh X \to X$ such that $p$ is biholomorphic when restricted to $X \setminus W$.}
Then we simply consider a K\"ahler metric with conic singularities on $\wh X$.
Its direct image, modified on $U$ by the usual maximum procedure, will be our desired $\omega_{\rm orb}$.
\medskip

\noindent Then we obtain the following more precise version of Theorem \ref{HE}.

\begin{thm}\label{HEquot} Let $X$ be a klt K\"ahler space and
  let $\cF$ be a stable $\Q$-sheaf on $(X, \omega_{\rm orb})$. There exists a metric $h_{\cF}= h_{\cF, 0}\exp(s)$ on $\cF|_{X_0}$ with the following properties:
\begin{enumerate}
\smallskip

\item[\rm (i)] It is smooth and it satisfies the Hermite-Einstein equation on $\displaystyle (X_{0}, \omega_{\rm orb}|_{X_0})$. Moreover, it has quotient singularities when restricted to
$X\setminus U$.  
\smallskip

\item[\rm (ii)] There exists a constant $C>0$ such that the inequalities
$$\Tr \exp(s) \leq C,\qquad \exp(s)\geq  Ce^{N\phi_Z}\Id_{\cF}$$
hold pointwise on $X_0$, where $N> 0$ is a positive, large enough constant.
\smallskip

\item[\rm (iii)] We have $$\int_{X_0}
\frac{1}{-\phi_Z}|D's|^2dV\leq C, \quad \int_{X_0} |\Theta(\cF, h_{\cF})|^2dV< \infty$$
where the volume element $dV$ is induced by the metric $\omega_{\rm orb}$  and the norm $|\cdot|$ is with respect to $h_{\cF, 0}$.
\end{enumerate}
\end{thm}

\noindent In particular, the Hermite--Einstein metric $h_{\cF}$ becomes non-singular when pulled back via the local uniformization maps providing the orbifold structure of $X \setminus U$.
This is the main reason for which we are using the metric $\omega_{\rm orb}$:
we show that the metric $h_{\cF}$ satisfying the estimates in Theorem~\ref{HEquot} and the Hermite--Einstein equation with respect to $\omega_{\rm orb}$ is automatically smooth in orbifold sense, cf.~Theorem~\ref{reg}.
\smallskip

\noindent Theorem \ref{HEquot} is proved along the same lines as Theorem \ref{HE}. However, we have to deal with an additional layer of 
difficulties: we are obliged to work with the metric $\omega_{\rm orb}$ (whose inverse image 
on a desingularization of $X$ and $\cF$
has both poles and zeroes), for the reasons just mentioned. Then, prior to applying the result \cite{Phong} we have to construct 
a family of smooth metrics approximating (the pull-back of) $\omega_{\rm orb}$, such that the hypothesis in \cite{Phong} are satisfied.
This is done by using an appropriate family of Monge-Ampère equations, cf. Sections 6 and 7. 
\smallskip

\begin{remark}
Assume that $X$ is projective, and let $H_1,\dots H_{n-2}$ be a set of hyperplane sections, such that $S:= \cap H_i$ has only quotient singularities
and $S\cap W= \emptyset.$ Then we can choose our set $U$ such that it does not intersects $S$. By Theorem \ref{HEquot}, the resulting metric $h_\cF$ restricted to $S$ has quotient singularities, i.e. it becomes smooth by pull-back via the local uniformisations. 
The conclusion is that the metric $h_\cF$ we construct in Theorem \ref{HEquot} matches the requirements in algebraic geometry, cf. e.g. \cite{GK20}.
\end{remark}
\medskip

\noindent We present next a few applications of Theorems \ref{HE} and \ref{HEquot}. Given a coherent, torsion-free sheaf $\cF$ on a compact Kähler space with at most klt singularities, one can 
consider a metric $h$ as in \eqref{mt2} and define the Chern forms $\theta_k(\cF, h)$ of $\cF|_{X_0}$ by the usual Chern-Weil formulas
(so that in the non-singular context $\theta_k(\cF, h)$ gives a representative for the $k^{\rm th}$ Chern class of $\cF$). In particular, let 
$$\Delta(\cF, h):= \theta_2(\cF, h)- \frac{r-1}{2r}\theta_1^2(\cF, h)$$ 
be the discriminant $(2, 2)$--form defined on $X_0$. 
\medskip

\noindent In Section 5 we establish the following result.
\begin{thm}\label{Chern1} We assume that $(X, \omega_X)$ and $\cF$ are satisfying the hypothesis of Theorem \ref{HE}. Let $h_{\cF}$ be the metric on $\cF$ which verifies the 
HE equation with respect to the metric $\omega_X$. Then the following hold:
\begin{enumerate}
\smallskip

\item[\rm (1)] For any smooth $(n-2, n-2)$-form $\eta$ on $X$ the (improper) integral
$$\int_{X\setminus Z}\Delta(\cF, h_{\cF})\wedge \eta< \infty$$
is convergent.
\smallskip

\item[\rm (2)] Assume that $\dim(X)= 3$ and let $\tau$ be any $(1, 0)$ smooth form with support in $X\setminus W$. Then the equality
$$\int_{X\setminus Z}\Delta(\cF, h_{\cF})\wedge \dbar \tau= 0$$ 
holds.
\smallskip

\item[\rm (3)] Assume that $\dim(X)= 3$. Then we have 
$$\int_{X\setminus Z}\Delta(\cF, h_{\cF})\wedge (\omega_X+ \ddc\varphi)\geq 0$$
where $\varphi$ is any smooth function on $X$ such that $\omega_X+ \ddc\varphi$ is equal to zero locally near $W$.
\end{enumerate}

\end{thm}
\smallskip

\begin{remark} The shortcoming of this result is that we ignore for the moment wether the linear map 
induced by $\Delta(\cF, h)$ on the space of forms with compact support on $X\setminus W$ thanks to point (2) depends on $h$. \end{remark}

\begin{remark} We believe that the equality in (2) above should hold regardless to the dimension of $X$. 
\end{remark}
\medskip



\noindent Anyway, we show that Theorem~\ref{HEquot} can be used in order to establish the following statement, which represents a particular case of the conjecture formulated in~\cite[Conj.~0.1]{CHP22}.
\begin{thm}\label{Chern2} Let $X$ be a compact klt Kähler threefold. We assume that $X$ admits a partial desingularisation as in \cite{LT19}. 
\footnote{This is claimed in the recent article \cite{O}}
Then the inequality at the point $(3)$ above holds true if the discriminant is defined by using the metric constructed in Theorem \ref{HEquot}.
\end{thm}

\noindent The proof of Theorem \ref{Chern2} can be obtained along the following lines. By hypothesis, we have a partial resolution 
$\pi: \wt X\to X$ of $X$, so that $\wt X$ has at most orbifold singularities, and $\pi$ is an isomorphism over a set whose complement is of codimension at least three in $X$. There exists a reflexive sheaf $\wt \cF$ which is isomorphic with the $\pi$-inverse image of $\cF$ 
and which is a $\mathbb Q$-sheaf \emph{in the complement of a finite set of points of $\wt X$}. 
Moreover, $\wt \cF$ is stable with respect to a well-chosen family of orbifold metrics, say $(\omega_\ep)_{\ep> 0}$.  
The main point is that, as consequence of Theorem \ref{HEquot}, we infer that for each $\ep$, the Chern classes inequality holds --and a limiting argument, allows us to 
conclude. Of course, if for some reason $\wt \cF$ has a $\mathbb Q$-sheaf structure on the entire $\wh X$, the sought-after inequality follows from the orbifold version of Uhlenbeck-Yau theorem, cf. \cite{Fau}, combined with the interpretation of the orbifold Chern classes in \cite{CGG}.
\smallskip

\noindent The rest of our article is organised as follows.

\tableofcontents

\medskip

\section{Mean-value inequality}\label{MV}

\noindent In this section we first establish the connection between the metric $h_{\cF, 0}$ constructed in Section \ref{intro}
and the desingularisation result in \cite{Ros68}.
Then we
recall a fundamental theorem due to \cite{Phong}, which is the main 
technical tool for the proof of the aforementioned results.
\medskip

\subsection{Desingularisation of sheaves} Let $\cF$ be a coherent sheaf on $X$. 
Our first results here are dealing with positivity properties of the metric $h_{\cF, 0}$ in \eqref{mt2}.
\begin{lemme}\label{prel1}
We denote by $\langle, \rangle_i$ and $\langle, \rangle_j$ the scalar products defined by \eqref{mt2}, corresponding to 
$\cA_i$ and $\cA_j$, respectively. There exists a positive constant $C_{ij}> 0$ such that 
for any section $s$ of $\displaystyle \cF|_{\cA_i\cap \cA_j}$ the following relation
\begin{equation}\label{comp1}
C_{ij}^{-1}< \frac{|s|_{i}}{|s|_{j}}< C_{ij}
\end{equation}
holds (modulo shrinking a little the trivialising sets $\cA_i$).
\end{lemme} 

\noindent The main point here is that we can compare the metrics above despite of the non-compactness of $X_0$.

\begin{proof} We recall that for each index $i$ metric $\displaystyle h_{\cF, i}$ is defined
via the formula
\begin{equation}\label{re7}
\langle s_1, s_2\rangle_i:= \langle \sigma_1, \sigma_2\rangle
\end{equation}
where $s_\alpha$ are sections of $\cF|_{\cA_i}$ and $\sigma_\alpha$ are sections of ${\mathcal O}^q_{\cA_i}$, for $\alpha= 1,2$ such that
\begin{equation}\label{ross1}
\beta(\sigma_\alpha)= s_\alpha, \qquad \langle \sigma_\alpha, \tau\rangle= 0
\end{equation} 
for any $\tau \in \ker(\beta)$. The scalar product on the RHS of \eqref{re7} is the usual one. 
\smallskip

\noindent When restricted to $\cA_{i, 0}= \cA_{i}\cap X_0$, the expression \eqref{re7} is point-wise 
defined, and it equals the quotient metric on $\displaystyle V_0|_{\cA_{i, 0}}$. 
In order to verify the claim of our lemma, let 
\begin{equation}\label{ross8}
\beta_k: {\mathcal O}^{q_k}_{\cA_k}\to \cF|_{\cA_k}
\end{equation} 
be the projections \eqref{mt1} corresponding to $k= i, j$. We can assume that there exists a morphism 
\begin{equation}\label{ross9}
\gamma: {\mathcal O}^{q_j}\to {\mathcal O}^{q_i}, \qquad \beta_i\circ \gamma= \beta_j
\end{equation}  
on the intersection $\cA_i\cap \cA_j$. Let $\sigma_j$ be a section of ${\mathcal O}^{q_j}$ such that $\beta_j(\sigma_j)= \xi$,
and such that moreover $\sigma_j$ is orthogonal to the kernel of $\beta_j$ (at the point we want to establish our inequality).
According to the definition, we have 
\begin{equation}\label{ross10}
|\xi|_j= |\sigma_j|.
\end{equation} 
On the other hand, there exists a positive constant $C> 0$ such that we have
\begin{equation}\label{ross11}
|\gamma (\tau)|\leq C|\tau|
\end{equation}  
(it is at this point that we might have to shrink the sets $\cA_i$) so all in all we get
\begin{equation}\label{ross12}
|\xi|_j= |\sigma_j|\geq  C^{-1}|\gamma (\sigma_j)|\geq C^{-1}|\xi|_i,
\end{equation}
where the last equality holds by the second part of \eqref{ross9}.
\end{proof}
\medskip

\noindent As consequence, we obtain the following result
\begin{proposition}\label{reflex} Let $\cF$ be a reflexive sheaf on $X$. Then it admits a metric $h_{}$ which is smooth on $X_0$ and such that
$$\sqrt{-1}\Theta(\cF, h_{})\geq -C\omega_X\otimes \Id_\cF$$
holds true at each point of $X_0$, where $C> 0$ is a positive constant. 
\end{proposition}
\begin{proof} The dual sheaf $\cF^\star$ is equally coherent, so we can
consider the metric $h_\alpha^\star$ induced on $\displaystyle \cF^\star|_{\cA_\alpha}$ by the dual of $\langle, \rangle_\alpha$. Griffiths's formula implies that we have 
\begin{equation}\label{re1}
\sqrt{-1}\Theta(\cF^\star, h^\star_{\alpha})\leq 0
\end{equation}
at each point of the intersection $\cA_\alpha\cap X_0$. 

\noindent By Lemma \ref{prel1}, given any pair of overlapping sets $\cA_\alpha\cap \cA_\beta$ there exists a positive constant $C_{\alpha\beta}> 1$ such that 
\begin{equation}\label{re2}
C_{\alpha\beta}^{-1}< \frac{|\xi|_{h_\alpha^\star}}{|\xi|_{h_\beta^\star}}< C_{\alpha\beta}
\end{equation}
for any local holomorphic section $\xi$. 

\noindent This happens to be precisely the condition we need in order to show the existence of a positive constant $C> 0$ such 
that the inequality 
\begin{equation}\label{re3}
\sqrt{-1}\ddc\log\left(\sum_i\theta_i^2|\xi|_{h_i^\star}^2\right)\geq -C\omega_X,
\end{equation}
cf. \cite{Dem92}, Lemma 3.5.
\smallskip

\noindent We denote by $h$ the metric induced on $\cF= \cF^{\star\star}$ by the dual of the metric obtained by glueing the $(h_\alpha^\star)$ by the partition of unit as in \eqref{re3}. The proposition is proved, since \eqref{re3} is equivalent to the inequality we want to establish.
\end{proof}
\medskip

\subsubsection{Connection with a theorem of Rossi}\label{Ross} We discuss next the link between the results above and the "desingularisation" of $\cF$, i.e. Theorem 3.5 in \cite{Ros68}. To this end, we first will briefly recall the general context in \emph{loc. cit.}
\smallskip

\noindent We consider an open cover $(\cA_\alpha)_\alpha$ of $X$, such that we have the exact sequence
\begin{equation}\label{re5}
{\mathcal O}^p_{\cA}\to {\mathcal O}^q_{\cA}\to \cF|_{\cA}\to 0
\end{equation} 
for each restriction of $\cF$ to $\cA$, where $\cA$ is one of the $\cA_\alpha$ above (of course, the integers $p$ and $q$ depend on $\alpha$).

\noindent Then we can construct a rational map
\begin{equation}\label{re4}
f: \cA\dashrightarrow \mathcal G(q-r, q) 
\end{equation} 
into the Grassmannian of $q-r$ planes in $\CC^q$, simply given by the kernel of the projection $$\beta: {\mathcal O}^q_{\cA}\to \cF|_{\cA}$$ of
\eqref{re5} on the open subset of $\cA$ where the dimension of the said kernel is $q-r$. 
\smallskip

\noindent On the manifold $\mathcal G(q-r, q)$ we have the following tautological exact sequence of vector bundles
\begin{equation}\label{re6}
0\to \xi_{q-r}\to \CC^q\to \eta_r\to 0
\end{equation} 
where the fiber of $\xi_{q-r}$ at a point of the Grassmannian $\mathcal G(q-r, q)$ is the subspace of $\CC^q$ corresponding to that 
point, and $\eta_r$ is the quotient bundle.
\medskip

\noindent The following results are established in \cite{Ros68}.

\begin{enumerate}
\smallskip

\item[(i)] Let $\displaystyle V_0:= \cF|_{X_0}$ be the locally free sheaf induced by $\cF$ on $X_0$. Then we have 
$$
V_0|_{\cA_{0}}= f^\star \eta_r,
$$
where we we use the notation $\cA_{0}:= \cA\cap X_0$.
\smallskip

\item[(ii)] For each index $i$ we denote by $\Gamma_i\subset \cA_{i}\times \mathcal G(q-r, q)$ the graph of the rational map $f_i$, cf. \eqref{re4},
equipped with the two projections $\pi_{i1}$ and $\pi_{i2}$. 
Then the several analytic sets $\Gamma_i$ piece together and let $\wt X$ be the resulting compact space. This is equipped with a map
$$
\pi:\wt X\to X
$$
induced locally by the $\pi_{i1}$. Moreover, $\pi$ is a modification. More precisely, Rossi shows that there exists an isomorphism
$$\theta_{ij}:\Gamma_{ij}\to \Gamma_{ji}$$
where $\displaystyle \Gamma_{ij}:= \pi_{i1}^{-1}(\cA_{ij})$ satisfying the cocycle relation\footnote{this is not stated explicitly, but it is a direct consequence of the arguments in \cite{Ros68}}, so that $\wt X$ is defined by the usual quotient of the disjoint union of the 
$\Gamma_i$'s. 
\smallskip

\item[(iii)] The sheaf 
$$E:= \pi^\star(\cF)/\Tor$$ is locally free on $\wt X$, and it is induced locally by the inverse image bundle $\displaystyle \pi_{i2}^\star (\eta_r)$.
\end{enumerate}
\medskip

\medskip

\noindent On the other hand, let $h_i$ be the quotient metric on the vector bundle $\eta_r$ in \eqref{re6}, where the 
flat bundle $\CC^q$ over the Grassmannian is endowed with the trivial metric. Then we have 
\begin{equation}\label{ross3}
h_{\cF, i}= f_\alpha^\star(h_i)
\end{equation} 
on $\cA_{i}\cap X_0$ by the definition of these objects. Dualising and pulling back \eqref{ross3} on the graph $\Gamma_i$ the equality
\begin{equation}\label{ross4}
\pi_{i 1}^\star h_{\cF, i}^\star= \pi_{i 2}^\star(h_i^\star)
\end{equation}
is satisfied for each index $i$, and given (ii) and (iii) above, it follows that we have the equality 
\begin{equation}\label{ross5}
\sum_\alpha \pi_{i1}^\star (\theta_i^2 h_{\cF, i}^\star)= \sum_i (\theta_i^2\circ\pi_{i1}) \pi_{i 2}^\star(h_i^\star)
\end{equation}
pointwise in the complement of a global analytic subset of $\wt X$. 
\smallskip

\noindent The expression on the RHS of \eqref{ross5} clearly extends as a smooth metric $\displaystyle h_{E^\star}$ on the vector bundle $E^\star$. The equality 
\eqref{ross5} combined with \cite{Dem92} shows that the following holds
\begin{equation}\label{ross6}
\sqrt{-1}\ddc\log|\rho|_{h_{E^\star}}^2\geq - C\pi^\star(\omega_X)
\end{equation}
for any local holomorphic section $\rho$ defined near an arbitrary point 
in the complement of an analytic subset of $\wt X$. 
\smallskip

\noindent Let $h_E$ be the metric on $E$ induced by $\displaystyle h_{E^\star}$; as consequence of \eqref{ross6}
we get 
\begin{equation}\label{ross2}
\sqrt{-1}\Theta(E, h_E)\geq -C\pi^\star(\omega_X)\otimes \Id_E
\end{equation}
in the complement of an analytic subset of $\wt X$ and therefore globally since the objects involved are smooth. Moreover, we denote 
that $\wt X$ a desingularisation of $\wt X$ and use the same notation $E$ for the pull-back of the bundle defined at (ii). 
\medskip

\noindent In conclusion, we have established the following result.

\begin{lemme}\label{lowerbound}
Let $\cF$ be a reflexive sheaf on a normal, compact analytic space $X$. We consider the map $\pi:\wh X\to X$ together with the locally free sheaf $E:= \pi^\star(\cF)/\Tor$ given by the theorem of Rossi. Then $E$ admits a Hermitian metric $h_E$ such that \eqref{ross2} holds.
\end{lemme}
\noindent The remarkable fact in \eqref{ross6} is that we obtained a \emph{uniform lower bound} for the curvature on $E$ as function of the inverse image of $\omega_X$.
\medskip

\begin{remark}\label{comparison} 
Let $s\in H^0\big(V, \cF|_V \big)$ be a local section of  
$\cF$ defined on some open subset $V\subset X$. Then for any $V'\Subset V$ there exists a positive constant $C(s)> 0$ such that we have 
\begin{equation}\label{ross81}
\sup_{V'_0} |s|_{h_{\cF, 0}}\leq C(s).
\end{equation}
Indeed this is the case for the metric $h_{\cF, i}$ since the norm of any projection is smaller than one. The metric $h_{\cF, 0}$ is constructed
by gluing together the local objects $(h_{\cF, i})_{i\in I}$ so \eqref{ross81} follows. 
\end{remark} 

\begin{remark}\label{comparison, 2} A natural question at this point would be to compare the metrics $h$ constructed in Proposition \ref{reflex} and $h_E$. 
Given the equality \eqref{ross3}, the answer is 
immediate. Let $s\in H^0\big(U, \cF|_U)\big)$ be a local section of  
$\cF$ defined on some open subset $U\subset X$. It induces a section $\sigma$ of $E|_{\pi^{-1}(U)}$ by taking the inverse image $\pi^\star (s)$ and then we have
\begin{equation}\label{mt5}
C^{-1}|s|_h^2\circ \pi\leq |\sigma|_{h_{E}}^2\leq {C}|s|_{h}^2\circ \pi
\end{equation}
point-wise on $\pi^{-1}(U)$.
\end{remark}

\begin{remark}\label{warning}
In general the direct image of $E$ is not equal to $\cF$, cf. \cite{RuSe17}.
\end{remark}


\subsection{Mean-value inequality} On the desingularization $\wh X$ of $(X, \cF)$ we consider the following "standard" metric. Let $D:= \sum D_i$ be the exceptional divisor of the map $\pi$ and let $\eta_i$ be any smooth representative of 
$c_1\big({\mathcal O}(-D_i)\big)$. There exists $\ep_0:= (\ep^1,\dots \ep^N) $ with $\ep^i >0$ such that the form 
$\omega_{\wh X}$ given by the formula
\begin{equation}\label{mt6}
\omega_{\wh X}:= \pi^\star\omega_X+ \sum_{i=1}^N \ep^i\eta_i.
\end{equation}
 is a K\"ahler metric on $\wh X$.
\medskip

\noindent Consider the compact K\"ahler manifold $(\wh X, \omega_{\wh X})$. In the recent and very interesting article \cite{Phong} the authors B. Guo, D.H. Phong and J. Sturm establish a very general 
mean inequality which will be crucial for our next arguments. We extract next from their article the version we need here.
\smallskip

\noindent Let $\chi\geq 0 (= \pi^\star\omega_X)$ be any semipositive $(1,1)$-form on $\wh X$, which is positive definite in the complement of an analytic set. We consider a family of K\"ahler metrics 
\begin{equation}\label{new6}
\omega_t\in \{\chi+ t\omega_{\wh X}\}
\end{equation}
for $t\in ]0, 1]$ and let 
\begin{equation}
F_t:= \log\frac{\omega_t^n}{\omega_{\wh X}^n}
\end{equation}
be the so-called entropy function.
\smallskip

\noindent We assume that there exist $p> n$ and a positive constant $C> 0$ such that 
\begin{equation}\label{new7}
\int_{\wh X}\left|F_t\right|^p\omega_t^n\leq C
\end{equation}
for all $t\in ]0, 1]$. 
\medskip

\noindent The next important result was established in \cite{Phong} (see page 7 in \emph{loc.cit} \footnote{where it is -strangely- downgraded to "Lemma", maybe because the main techniques used to establish it are already present in the author's precedent work \cite{Phong1}}). 

\begin{thm}\cite{Phong} \label{phong}
Let $v\in L^1(\wh X, \omega_t^n)$ be a function which is $\cC^2$-differentiable on the set 
$(v\geq -1)$ and such that 
\begin{equation}
\Delta''_{t}(v)\geq -a
\end{equation} 
on $\wh X$, where we denote by $\Delta''_t$ the Laplace operator corresponding to 
$(\wh X, \omega_t)$. Then we have 
\begin{equation}\label{new9}
\sup_{\wh X}v\leq C_{\rm univ}(1+ \Vert v\Vert_{L^1(\wh X, \omega_t)})
\end{equation} 
where $C_{\rm univ}$ only depends on $p, n, \chi, \omega_{\wh X}$ and the bound $C$ for the entropy in \eqref{new7}.
\end{thm}

\begin{remark} We note that the additional normalisation $\displaystyle \int_{\wh X}v dV_t=0$ is required in \emph{loc. cit.}, where $dV_t$ is the volume element of $(X, \omega_t)$.  By replacing function $v$ in Theorem \ref{phong} with 
$$\displaystyle v- \frac{1}{\Vol_t(\wh X)}\int_{\wh X}v dV_t,$$ 
where $\Vol_t(\wh X)$ is the volume of $\wh X$ with respect to $\omega_t$, we see that this result holds even without the normalisation, 
since $\Vol_t(\wh X)$ is uniformly bounded (thanks to \eqref{new6} above). 
\end{remark}
\medskip

\noindent We will use this result in the following context. Let $\wh X$ be as above, and consider the semipositive form
$$\chi:= \pi^\star\omega.$$ 

\noindent Define the family of metrics 
\begin{equation}\label{new12}
\omega_{t}:= \pi^\star \omega_{X}+ t\omega_{\wh X},\end{equation} 
where $t\in ]0, 1[$ is a real number. The main point is that the associated entropy function
\begin{equation}\label{new11}
F_{t}:= \log\frac{\omega_{t}^n}{\omega_{\wh X}^n}
\end{equation}
verifies an inequality as required in \eqref{new7}, as we next show.
\begin{lemme}\label{entro1} For any $p> 0$ 
there exists a constant $C>0$ uniform with respect to the parameter ${t}$ such that 
\begin{equation}\label{new13}
\int_{\wh X}\left|F_{t}\right|^p\omega_{t}^n\leq C
\end{equation}
holds true for all $t\in [0, 1/2]$.
\end{lemme}
\begin{proof}
The verification of \eqref{new13} is very simple, as follows. In the first place we can assume 
that $\omega_t\leq \omega_{\wh X}$ for each $t\in [0, 1/2]$, so that we have 
\begin{equation}\label{new30}
|F_{t}|= \log\frac{\omega_{\wh X}^n}{\omega_{t}^n}.
\end{equation}
\noindent Next, when restricted to 
each of the sets $\mathcal A_\alpha\subset U_\alpha$ the metric $\omega_X$ is quasi-isometric 
to the Euclidean metric on $\CC^{N_\alpha}$ (notations as at the beginning of this section). It follows that 
locally on each trivialising open subset $\Omega$ of $\wh X$ there exist holomorphic functions 
$$\displaystyle f_\alpha:= (f_{\alpha, 1},\dots, f_{\alpha, M})$$
for ${\alpha=0,\dots, n}$ such that 
\begin{equation}\label{new31}
|F_{t}|_{\Omega}\leq -\log{\sum_{\alpha=0}^n\sum_it^\alpha |f_{\alpha, i}|^2}
\end{equation}
and such that $f_0$ is not identically zero. Since the RHS of the inequality belongs to $L^p$ for any positive $p> 0$, the lemma follows.
\end{proof}

\medskip

\noindent Summing up, we have the next statement.
\begin{cor}\label{mv}
$v\in L^1(\wh X, \omega_{t}^n)$ be a function which is $\cC^2$-differentiable on the set 
$(v\geq -1)$ and such that 
\begin{equation}\label{new14}
\Delta''_{t}(v)\geq -a
\end{equation} 
on $\wh X$, where we denote by $\Delta''_{t}$ the Laplace operator corresponding to 
$(\wh X, \omega_t)$. Then we have 
\begin{equation}\label{new15}
\sup_{\wh X}v\leq C_{\rm univ}\big(1+ \Vert v\Vert_{L^1_{t}}\big)
\end{equation} 
where $C_{\rm univ}$ only depends on the quantities in Theorem \ref{phong} -in particular, is independent of the parameter $t$. The $L^1$ norm in \eqref{new15} is induced by the data $(\wh X, \omega_{t})$. 
\end{cor}
\medskip

\section{Approximate stability and uniformity properties}
\medskip

\noindent We now assume that the sheaf $\cF$ is stable with respect to the data $(X, \{ \omega_{X} \})$.
Thanks to the results in~\cite{Tom19}, the corresponding bundle $E$ is stable with respect to $\omega_{t}$, provided that the parameter $t$ is small enough.
Actually, the result in~\cite{Tom19} is very general; we discuss next a particular case in which it can be easily established directly.
\medskip

\noindent \emph{Hypothesis}:\label{hypo} Each class $\displaystyle \{ \alpha \} \in H^{1,1}(\wh X, \R) \cap H^2(\wh X, \Z)$ is the inverse image of the Chern class of a line bundle on $X$, modulo a cohomology class with support in $D$.
\smallskip

\noindent \emph{Remark:} If $X$ is algebraic (but not necessarily compact) and has rational singularities, then the \emph{Hypothesis} is equivalent to $X$ being $\Q$-factorial, cf.~\cite[Prop.~12.1.6]{KollarMori92}.
This equivalence should continue to hold without the algebraicity assumption, as long as $X$ is compact (which it is in our situation).
\smallskip


%

\noindent Assuming the \emph{Hypothesis} from now on, we argue as follows.
In the first place it is clear that $E$ is stable with respect to $\pi^\star\omega_X$- this follows from the stability of $\cF$. We show next that there exists a positive $\delta_0>0$ such that 
\begin{equation}\label{new16}
\mu_{\pi^\star\omega_X}(\cG)+ \delta_0< \mu_{\pi^\star\omega_X}(E)
\end{equation}
holds for every proper sub-sheaf $\cG\subset E$. As soon as the existence of $\delta_0>0$ is granted, the conclusion follows immediately.
\smallskip

\noindent If such a constant $\delta_0> 0$ is impossible to find, then there exists a 
sequence 
\begin{equation}\label{new32}
0\to\cG_k\to E\end{equation} 
such that $\rk(\cG_k)= s< r:= \rk(E)$ and such that
 \begin{equation}\label{new17}
\lim_{k\to\infty}\int_{\wh X}c_1(L_k)\wedge \pi^\star\omega_X^{n-1}= \frac{s}{r}
\int_{\wh X}c_1(E)\wedge \pi^\star\omega_X^{n-1}
\end{equation}
where $L_k$ is the (double dual of the) determinant of $\cG_k$. 
\smallskip

\noindent Let 
$\sigma_k$ be the section of $\Lambda^s E\otimes L_k^{-1}$ induced by the map \eqref{new32}.
We have 
\begin{equation}\label{new18}
\sqrt{-1}\ddbar \log|\sigma_k|^2 _{h_E, h_k}\geq \theta_{L_k}- C\omega_{\wh X}
\end{equation}
where $\theta_{L_k}$ is the curvature of $L_k$ with respect a fixed smooth metric $h_k$ and $C$ is a positive constant bounding the curvature of $\displaystyle (E, h_{E,0})$. The current 
\begin{equation}\label{new19}
\Theta_k:= \pi_\star\big(C\omega_{\wh X}- \theta_{L_k}+ \sqrt{-1}\ddbar \log|\sigma_k|^2\big)
\end{equation}
is positive, and its mass is bounded by a constant which is independent of $k$: indeed, as consequence of
\eqref{new17} we see that 
\begin{equation}\label{new20}
\int_X\Theta_k\wedge \omega_X\leq C
\end{equation}
for some constant $C>0$. It follows that 
the sequence of currents $(\Theta_k)_{k\geq1}$ is convergent (up to the choice of some subsequence). 
\smallskip

\noindent Now, by our hypothesis we have 
\begin{equation}\label{coho3}
c_1(L_k)\simeq \pi^\star(F_k)
\end{equation}
modulo a $(1,1)$-class whose support is included in $D$. 
Then we infer that 
the sequence of cohomology classes $\{\theta_{F_k}\}:= \pi_\star \{\theta_{L_k}\}$ is convergent --therefore constant, as soon as $k\gg 0$. Therefore the sequence of real numbers on the LHS of \eqref{new17} is constant, and this contradicts the stability of $E$.
\smallskip

\noindent In the absence of the \emph{Hypothesis} above, we simply use the result in \cite{Tom19}.

\medskip 

\noindent By the theorem of Uhlenbeck-Yau \cite{UY86} applied for each $t$ we can construct a sequence $(h_{t, \infty})_{t}$ of 
metrics on $E$, such that $h_{t, \infty}$ satisfies the Hermite-Einstein equation with respect to 
the metric $\omega_{t}$ in \eqref{new12}. Next we choose a specific normalisation of $h_{t, \infty}$ and our plan is to show that by letting $t\to 0$ the limit metric will satisfy the requirements in Theorem \ref{HE}. 
\smallskip

\noindent The normalisation mentioned before is as follows. Let $(L, h_L):= (\det E, \det h_E)$ et let $\theta_L$ be the induced curvature form on $L$. We recall that the metric $h_E$ satisfies the inequality \eqref{ross2}.
The equation
\begin{equation}\label{eq4}
\Lambda_{t} \theta_L+ \Delta_{t}''\rho_{t}= 
\frac{1}{\Vol(X, \omega_{t})}\int_{\wh X}c_1(E)\wedge \{\omega_{t}\}^{n-1}
\end{equation} 
has a unique solution $\rho_{t}$ such that $\displaystyle \int_{\wh X}\rho_{t}dV_t= 0.$ This specific normalisation, rather than 
$\displaystyle \max_{\wh X}\rho_{t}= 0$ will be important in the proof of our main results.
\medskip

\noindent As an application of the Corollary \ref{mv} combined with \eqref{eq4}
we get the following important information.
\begin{lemme}\label{mv1}
There exists a positive constant $C>0$ such that the inequalities
\begin{equation}\label{new21}
- \rho_{t}\leq C\big(1+ \int_{\wh X}|\rho_{t}|dV_{t}\big), \qquad \sup_{\wh X}\left(C\log |s_D|^{2}+ \rho_{t}\right)\leq \int_{\wh X}|\rho_{t}|dV_{t}, 
\end{equation}
hold true for every $t\in ]0, 1]$. 
\end{lemme}
\begin{proof}
We begin with the following inequalities
\begin{equation}
\theta_L\geq -C\pi^\star \omega_{X}\geq  -C\omega_t
\end{equation}
which are a direct consequence of \eqref{ross2}.
\smallskip

\noindent By taking the trace with respect to the metric $\omega_{t}$ it follows that we have 
\begin{equation}\label{new23}
\Lambda_{t} \theta_L\geq -C.
\end{equation}
The inequality \eqref{new23} combined with \eqref{eq4} implies that we have
\begin{equation}\label{new24}
\Delta_{t}''\left(- \rho_{t}\right)\geq -C
\end{equation}
and the first inequality of \eqref{new21} follows from Corollary \ref{mv}.
\medskip

\noindent The second part of our statement is established along the same type of arguments, modulo the fact that the form $\theta_L$ is not necessarily bounded from above by the inverse image of $\omega_X$. We proceed as follows:
\begin{equation}\label{rev1}
\theta_L\leq C\omega_{\wh X}= C(\pi^\star\omega_X+ \sqrt{-1}\ddbar \log|s_D|^{2})
\end{equation}
and since $\pi^\star\omega_X\leq \omega_t$ it follows that 
\begin{equation}\label{rev2}
\theta_L\leq C(\omega_t+ \sqrt{-1}\ddbar \log|s_D|^{2})
\end{equation}
and therefore we obtain the inequality 
\begin{equation}\label{rev3}
\Lambda_{t} \theta_L\leq C(1+ \Delta_{t}''\log|s_D|^{2\ep_0}).
\end{equation}
The relation \eqref{rev3} combined with the equality \eqref{eq4} and Corollary \ref{mv} allow us to conclude.
\end{proof}

\medskip

\noindent We can express the HE metric $h_{t, \infty}$ as follows
\begin{equation}\label{mt9} 
h_{t, \infty}:= h_{E}e^{- \frac{1}{r}\rho_t}\exp(s_t)
\end{equation}
where $s_t$ is an endomorphism of $E$ such that $\Tr(s_t)= 0$ at each point of $\wh X$. We wish to obtain $L^\infty$ estimates for the norm of the endomorphism 
\begin{equation} 
H_t:= e^{- \frac{1}{r}\rho_t}\exp(s_t)=: e^{- \frac{1}{r}\rho_t}S_t
\end{equation} 
which defines the HE metric $h_{t, \infty}$. Note that have the equality
\begin{equation}\label{rev6}
\log\big(\Tr (H_t)\big)= \log\big(\Tr (S_t)\big)- \rho_t.
\end{equation} 
\smallskip

\noindent Next we obtain an inequality similar to the one in Lemma \ref{mv1} for the function $\log \Tr(H_t)$. The method will be the same: derive a lower bound for 
$\Delta_{t}''\log \Tr(H_t)$ and use the corollary in the previous section. 
\medskip

\noindent Recall the following identities
\begin{equation}\label{rev40}
\Delta''_t\log\big(\Tr (H_t)\big)= \frac{\Delta''_t\big(\Tr (H_t)\big)}{\Tr (H_t)}- \frac{|\partial \Tr (H_t)|^2}{\big(\Tr (H_t)\big)^2}
\end{equation}
and 
\begin{equation}\label{rev41}
\Delta''\big(\Tr (H_t)\big)= \Tr \big(\Lambda_t\Theta(E, h_{E})\circ H_t\big)+ \Tr \Lambda_t\big((D'_tH_t)\circ H_t^{-1}\circ (\dbar H_t)\big)+ C
\end{equation}
 cf. \cite{Siu87}, pages 12-14, where the constant $C$ in \eqref{rev41} is given by the HE condition. As shown in the \emph{loc. cit}, we have 
 \begin{equation}\label{rev42}
\Tr \Lambda_t\big((D'_tH_t)\circ H_t^{-1}\circ (\dbar H_t)\big)\geq \frac{|\partial \Tr (H_t)|^2}{\Tr (H_t)}
\end{equation}
at each point of $X$ from which we obtain the relation
\begin{equation}\label{new25}
\Delta_{t}''\left(\log \Tr(H_t)\right)\geq \frac{\Tr\big(\Lambda_t \Theta(E, h_{E})\circ H_t\big)}{\Tr(H_t)}+ C.
\end{equation}
On the other hand, by using \eqref{ross2} again we infer that we have
\begin{equation}\label{new33}
\frac{\Tr\big(\Lambda_t \Theta(E, h_{E})\circ H_t\big)}{\Tr(H_t)} \geq -C
\end{equation}
pointwise on $\wh X$.

\noindent When combined with \eqref{new25}, it gives
\begin{equation}\label{new26}
\Delta_{t}''\left(\log \Tr(H_t)\right)\geq -C.
\end{equation}
By using Corollary \ref{mv} again, we have
\begin{equation}\label{rev7}
\log \Tr(H_t)\leq C\big(1+ \big\Vert\log\big(\Tr(H_t)\big)\big\Vert_t\big),
\end{equation}
a relation which will play an important role in the next section, so we record it here for further reference.
\begin{lemme}\label{1mv1}
There exists a positive constant $C>0$ such that the inequalities
\begin{equation}\label{new221}
\log \Tr(H_t)\leq C\big(1+ \big\Vert\log\big(\Tr(H_t)\big)\big\Vert_t\big),
\end{equation}
hold true for every $t\in ]0, 1]$. 
\end{lemme}

\medskip

\begin{remark} The main point in Lemma \ref{1mv1} is that the constant $C$ is \emph{independent of $t$}. By anticipating a little, as soon as we get a uniform $L^1(\wh X, \omega_t)$ 
bound for the function $\log\big(\Tr H_t\big)$, we are done (i.e. precisely as in the "usual case"). Actually the two 
statements Lemma \ref{mv1} and Lemma \ref{1mv1} represent the analogue of Simpson's properness theorem in \cite{Sim88}
for the sequence of metrics $(h_{t, \infty})_{t> 0}$.
\end{remark}

\section{Proof of Theorem \ref{HE}}\label{smoothset}

\noindent Given any endomorphism $s$ of a vector bundle $E$ together with functions $\psi: \mathbb R\to \R$ and 
$\Psi: \R\times \R\to \R$ we first recall here the definition of the canonical 
constructions $\psi(s)$ and $\Psi(s)$. Afterwards Theorem \ref{HE} will be 
established via a suitable modification of the usual 
techniques in Hermite-Einstein theory combined with the results in the previous section. 
\medskip

\noindent The main technical result obtained in this section is the following $\mathcal C^0$ estimate.

\begin{thm}\label{Est, I} There exists a constant $C> 0$ such that the inequalities 
\begin{equation}\label{est1}
\Tr (H_t)\leq C, \qquad -C\leq \rho_t\leq -C\log|s_D|^2\qquad \int_{\wh X}|D's_{t}|^2\frac{dV_{t}}{\log\frac{1}{|s_D|^2}}\leq C
\end{equation}
hold, where $s_D$ is the product of sections defining the exceptional divisor of the map $\pi: \wh X\to X$.
\end{thm}
\medskip

\noindent We remark that after establishing Theorem \ref{Est, I},  Theorem \ref{HE} will be an immediate consequence. To begin with,  
we recall the following.

\subsection{Canonical constructions} Recall that the vector bundle $E$ is endowed with a fixed, reference metric $h_E$ and let $s\in \End(E)$
be a Hermitian endomorphism of $E$. Let $\psi: \R\to \R$ and 
$\Psi: \R\times \R\to \R$ de two smooth functions. They induce maps
\begin{equation}\label{can1}
\psi: \End(E)\to \End(E), \qquad \Psi: \End(E)\to \End\big(\End(E)\big)
\end{equation}
as follows. Locally on each co-ordinate subset $U\subset \wh X$ we have an orthonormal frame $e_1,\dots, e_r$ of $E$ with respect to which $s$ is diagonal, say $s(e_i)= \lambda_i e_i$ for some real functions $\lambda_i$ defined on $U$. Then we set $\psi(s)(e_i):= \exp(\lambda_i)e_i$ for each $i= 1,\dots, r$ and we obtain in this way a -global- endomorphism $\psi(s)$. We note that this can also be defined by the more intrinsic formula
\begin{equation}\label{can2}
\psi(s)= \frac{1}{2\pi}\int_{\R}\wh \psi(\tau)\exp(\sqrt{-1}\tau s)d\tau
\end{equation}
where $\wh \psi$ is the Fourier transform of a suitable truncation of $\psi$ (obtained by taking into account the size of the eigenvalues of $s$).
\smallskip

\noindent Given a $(p, q)$-form $\End(E)$-valued $A$ we can locally write $\displaystyle A= \sum_{i,j} a^i_j e_i\otimes e^j$ where the coefficients $a^i_j$ are forms of type $(p, q)$ on $U$ and where $e^1,\dots, e^r$ is the dual basis on $E^\star$ induced by 
$e_1,\dots, e_r$. Then we define
\begin{equation}\label{can3}
\Psi(s)(A)|_U:= \sum_{i,j} \Psi(\lambda_i, \lambda_j)a^i_j e_i\otimes e^j 
\end{equation}and again this defines a global endomorphism of $E$. 

\noindent For the main properties of $\psi(s)$ and $\Psi(s)(A)$ we refer to \cite{Sim88} and the appendix in [LüTe]. An important property is that if $\Psi (\lambda_1, \lambda_2) >0$ for all $\lambda_1, \lambda_2$, then $\langle \Psi (s) A, A\rangle \geq 0$. It follows from \eqref{can3}.  

\noindent Moreover, we consider $\displaystyle \Psi (x, y) =\frac{e^{x-y}-1}{x-y}$ and we define $\lambda_{min}, \lambda_{max}$ as min and max of the eigenvalues of $s$. Set $a:= \lambda_{max} - \lambda_{min}$.
Then \eqref{can3} implies that the inequality 
\begin{equation}\label{ineqpsi}
	\langle \Psi (s) A, A\rangle \geq \frac{1- e^{-a} }{a} |A|^2,
\end{equation}	
holds true, a relation which will be very useful in the next subsection.
\smallskip

\noindent To end this subsection, we recall that the objects above have a natural extension in the context of Sobolev spaces. A section $\sigma$ of $\End(E)$
belongs to the space $W^{1, p}$ if $\sigma$ and $\dbar \sigma$ are in $L^p$.
For each positive real number $b> 0$  denote by $L^p_b$ and $W^{1, p}_b$
the space of sections in $L^p$ and $W^{1, p}$ respectively, whose $L^\infty$
norm is moreover bounded by $b$.
\medskip

\noindent In this context we have the following simple result.
\begin{lemme}\label{extensions} \cite{Sim88}
Let $S$ be the set of symmetric endomorphisms of $E$. We consider two functions $\psi$ and $\Psi$ as above. Then the following hold true:
\begin{enumerate}
  \smallskip

  \item The map $\psi$ extends to a continuous linear application
    $$\psi: L^p_b(S)\to L^p_{b'}(S)$$
    for some $b'> 0$.
\smallskip

  \item The map $\Psi$ extends to a linear application
    $$\Psi: L^p_b(S)\to \Hom\big(L^p(\End E), L^q(\End E)\big)$$
for any $q\leq p$ and it is moreover continuous in case $q< p$.     
\end{enumerate} 
\end{lemme}



\medskip

\subsection{End of the proof}  Our first target is to establish a differential inequality for the derivative $D's_t$, cf. lemma below. Along the following lines we evaluate the scalar product 
$\displaystyle \langle(D'S_t)S_t^{-1}, D's_t\rangle$ at each point of $\wh X$.

\noindent To this end we follow the method in \cite{Sim88}. Let $e_1,\dots, e_r$ be a local orthogonal frame of $(E, h_E)$ with respect to which $s_{t}$ is diagonal, so that we have

\begin{equation}
S_t= \sum\exp(\lambda_i)e_i\otimes e_i^\star
\end{equation}
for some functions $\lambda_i$ defined locally. Then 
 it follows that 
\begin{equation}\label{prop4}
(D'S_t)S_t^{-1}= \sum \partial \lambda_i e_i\otimes e_i^\star+ \sum (\exp(\lambda_i- \lambda_j)-1)a_{ij}e_i\otimes e_j^\star
\end{equation}
where the $a_{ij}$ are the coefficients of the connection matrix corresponding to the frame $e_1,\dots, e_r$. 

\noindent Therefore we have the equality
\begin{equation}\label{alt2}
\langle(D'S_t)S_t^{-1}, D's_t\rangle= \sum |\partial \lambda_i|^2_t + 
\sum (\lambda_i- \lambda_j)\big(\exp(\lambda_i- \lambda_j)- 1)\big)|a_{ij}|^2_t.
\end{equation}
as well as 
\begin{equation}\label{alt3}
\langle D's_t, D's_t\rangle= \sum |\partial \lambda_i|^2_t + 
\sum (\lambda_i- \lambda_j)^2|a_{ij}|^2_t
\end{equation}
which imply that 
\begin{equation}\label{alt4}
\langle(D'S_t)S_t^{-1}, D's_t\rangle \geq \langle \Phi(s_t)(D's_t), D's_t\rangle
\end{equation}
where $\displaystyle \Phi(x, y):= \frac{\exp(x-y)- 1}{x-y}$ (cf. notations in the previous section).
\smallskip

\noindent Our next step is the following statement.

\begin{lemme}\label{W1} Let $\eta_t:= \log H_t$ be the logarithm of the endomorphism defining the Hermite-Einstein metric. Then the inequality
\begin{equation}\label{alt5}
0 \geq  
\int_{\wh X}\left\langle \Phi(\eta_t)(D'\eta_t), D'\eta_t\right\rangle dV_t+ \int_{\wh X} \Tr \big(\eta_t \Lambda_{t}\Theta(E, h_E)\big)dV_t
\end{equation}
holds for every $t$.
\end{lemme}

\begin{proof}
The HE equation for the metric $h_{t, \infty}$ (notations as in Section $3$) writes
\begin{equation}\label{prop2}
\gamma_t \Id- \Lambda_{t}\Theta(E, h_E)= \Lambda_{t} \dbar\left((D'H_t)H_t^{-1}\right).
\end{equation}
We compose the equality \eqref{prop2} with the endomorphism $$\eta_t= -\frac{\rho_t}{r}\otimes \Id+ s_t,$$ take the trace and integrate over $X$. Since 
the mean of the trace of $\displaystyle \eta_t$ is equal to zero (this is where the normalisation of $\rho_t$ plays a role), 
we get
\begin{equation}\label{alt6}
0 =  
\int_{\wh X}\Tr \left(D^{\prime\star}\big((D'H_t)H_t^{-1}\big)\circ \eta_t\right)dV_t+ \int_{\wh X} \Tr \big(\eta_t \Lambda_{t}\Theta(E, h_E)\big)dV_t
\end{equation}
and we remark that the equality
\begin{equation}\label{alt7}
\int_{\wh X}\Tr \left(D^{\prime\star}\big((D'H_t)H_t^{-1}\big)\circ \eta_t\right)dV_t= 
\int_{\wh X} \langle(D'H_t)H_t^{-1}, D'\eta_t\rangle dV_t
\end{equation}
holds, since $\eta_t$ is self-adjoint.
\smallskip

\noindent In conclusion, the first integral on the RHS of \eqref{alt6} coincides with the LHS of \eqref{alt4} which ends the proof of our claim. A last remark is that we have $$D'\eta_t= (\dbar \eta_t)^\star,$$ so we can freely replace $D'$ with $\dbar$ in \eqref{alt6}. 
\end{proof}

\begin{remark} This lemma will be used exactly as in \cite{Sim88} in order to show the existence of a destabilising sheaf in case the $L^1$ norm of $H_t$ tends to infinity as $t\to 0$.
Arguments similar to \cite[pages 885-888]{Sim88} will show that if the 
$L^1$ norm of $H_t$ tends to $\infty$ as $t$ tends to zero, then we get a destabilising sub-sheaf of $E$, where the degree is measured with respect to  
$\displaystyle\pi^\star \omega_{X}$. 
\end{remark}

\medskip

\begin{remark} The metrics $h_{t, \infty}$ are already HE, so it might be possible to conclude by a more direct argument than the following one. For example, one could imagine an iteration scheme in order to bound the 
$L^1$ norm of $H_t$, by combining the mean-value inequality with Lemma \ref{W1}, but the problem is the
absence of a uniform Sobolev constant for $(\wh X, \omega_t)$. \footnote{In the meantime, i.e. after the first version of this article was released, 
the sought-after Sobolev inequality was established in \cite{GPSS} by B. Guo, D. H. Phong, J. Song and J. Sturm. Thus, there is some hope to avoid the case-by-case discussion that follows.}
\end{remark}

\medskip

\noindent Anyway, as $t$ tends to zero two --mutually exclusive-- cases can occur:

\begin{enumerate}
\smallskip

\item[(1)] There exist a constant $C> 0$ such that 
$$\int_{\wh X}\big(|\rho_t|+ \log {\Tr S_t}\big) dV_t \leq C$$
for all positive $t$ (note that $\log \Tr S_t\geq \log r$ since the trace of $s_t$ is equal to zero). 
\smallskip

\item[(2)] There exist sequences $(\delta_i)_{i\geq 1}$ and  $(t_i)_{i\geq 1}$ of positive numbers 
converging towards zero such that 
$$\int_{\wh X}\big(|\delta_i\rho_i|+ \delta_i \log {\Tr S_i}\big) dV_i = 1$$
where we recall that $r$ is the rank of $E$. We denote by $dV_i$ is the
volume element corresponding to the metric 
$\displaystyle \omega_{t_i}$. Also, $\rho_i$ and $S_i$ correspond to $\displaystyle \rho_{t_i}$ and $\displaystyle S_{t_i}$, respectively.
\end{enumerate}
\medskip

\subsubsection{The analysis of the first case} Thanks to the inequality \eqref{rev7}, the uniform estimate given in case (1) implies that we have
\begin{equation}\label{1stcase1}
\Tr H_t\leq C, \qquad -C\leq \rho_t\leq -C\log|s_D|^2
\end{equation}
for all positive $t\ll 1$, so that the first part of Theorem \ref{Est, I} is proved.
\smallskip

\noindent We combine \eqref{1stcase1} with the inequality \eqref{alt5},
so we get
\begin{equation}\label{1stcase3}
\int_{\wh X}\left\langle \Phi(\eta_t)(D'\eta_t), D'\eta_t\right\rangle dV_t
\leq
 C\int_{\wh X} \big(1+ \log\frac{1}{|\sigma_D|^2}\big) |\Lambda_{t}\Theta(E, H_0)|dV_t
\end{equation}

\noindent It is clear that we have $\displaystyle \sup_{\wh X}|\Lambda_{t}\Theta(E, H_0)|dV_t\leq CdV$ point-wise on $\wh X$, where $dV$ is the volume element corresponding to a fixed reference metric on $\wh X$ and $C$ is
a constant independent of $t$.

\noindent Moreover, the estimate \eqref{1stcase1} shows that the inequality
\begin{equation}\label{1stcase4}
\left\langle \Phi(\eta_t)(D'\eta_t), D'\eta_t\right\rangle 
\geq \frac{C}{\log\frac{1}{|\sigma_D|^2}}|D'\eta_t|^2
\end{equation}
holds, provided that the positive constant $C\ll 1$ is small enough.
\smallskip

\noindent All in all, we obtain
\begin{equation}\label{1stcase5}
\int_{\wh X}
\frac{1}{\log\frac{1}{|\sigma_D|^2}}|D'\eta_t|^2dV_t\leq C
\end{equation}
which ends the proof of Theorem \ref{Est, I}.

\noindent Combined with the fact that $h_{t, \infty}$ is Hermite-Einstein, we can take the limit on any compact subset of $\wh X\setminus D$ \footnote{this is standard, see e.g. \cite{BS94}.} and conclude that we have
\begin{equation}\label{1stcase6}
\int_{\wh X\setminus D}
\frac{1}{\log\frac{1}{|\sigma_D|^2}}|D'\eta_{\infty}|^2dV_{0}\leq C, \qquad
|\eta_{\infty}|\leq C+ \log\frac{1}{|\sigma_D|^2}.
\end{equation}
\medskip

\subsubsection{The analysis of the second case} We show that this cannot occur by modification of the usual arguments.
This is due to the presence of $\log|s_D|$ in the estimates for
$s_t$. Moreover, as opposed to the standard context, the background metric $\omega_t$ also varies. For all these reasons, we have decided to provide a complete treatment. 
\smallskip

\noindent Let $\displaystyle u_i:= \delta_i\eta_i= -\frac{\delta_i}{r}\rho_i\otimes \Id_E+ \delta_is_i$; in this case we get
\begin{equation}\label{rev9}
\sup_{\wh X}\big(-\delta_i\rho_i+ \delta_i\log\Tr(S_i)\big)\leq C
\end{equation}
and then it follows from Lemma \ref{mv1} that we have
\begin{equation}
\sup_{\wh X}\big(\delta_i|\rho_i|+ \delta_i\log\Tr(S_i)\big)\leq C(1- \delta_i\log|\sigma_D|^2).
\end{equation}
Indeed, we use the normalisation of $\rho_i$ together with the fact that $\Tr(S_i)\geq r$. Note that we have 
\begin{equation}\label{rev20}
\delta_i\log\frac{\Tr(S_i)}{r}\geq \log\frac{\Tr(S_i^{\delta_i})}{r}
\end{equation}
by concavity, so it follows 
that the inequality 
\begin{equation}\label{2ndcase2}
|u_i|\leq C+ \delta_i\log\frac{1}{|\sigma_D|^2}
\end{equation}
holds true on $\wh X$ (by using one more time that the trace of $s_i$ equals zero).
\medskip

\noindent The important step towards the contradiction we are looking for is the 
following result (analogue of Lemma 5.4 in \cite{Sim88}).
\begin{lemme}\label{weak} There exist a subsequence of $(u_i)_{i\geq 1}$ converging weakly to a limit
$u_\infty$ on compact subsets of $X\setminus D$ such that the following hold.
\begin{enumerate}
\smallskip

\item[\rm (1)] The endomorphism $u_\infty$ is non-identically zero and it belongs to the space $H^1$ (i.e. it is in $L^2$ together with its differential with respect to the pull-back metric $\pi^\star\omega_X$).
\smallskip

\item[\rm (2)] Let $\Psi:\R\times \R\to \R_+$ be a smooth, positive function such that $\displaystyle \Psi(a, b)< \frac{1}{a-b}$ holds for any
$a> b$. Then we have 
$$0 \geq  
\int_{\wh X}\left\langle \Psi(u_\infty)(D'u_\infty), D'u_\infty\right\rangle dV_0+ \int_{\wh X} \Tr \big(u_\infty \Lambda_{0}\Theta(E, h_E)\big)dV_0$$
where $\Lambda_0$ and $dV_0$ are the contraction with and the volume element corresponding to $\pi^\star\omega_X$, respectively. 
\end{enumerate}
\end{lemme} 

\begin{proof}
The main arguments are taken from \cite{Sim88}, pages 885-886. The slight difference is that instead of being uniformly bounded,
the sequence $(u_i)$ only satisfies \eqref{2ndcase2}. Nevertheless, it is still true that given $\ep> 0$ there exists $|D|\subset U_\ep\subset X$ 
a small enough open subset $U_\ep$ containing the support of $D$ such that 
\begin{equation}\label{new38}
\int_{U_\ep}|u_i|dV_X< \ep
\end{equation}
for any index $i$. This is an immediate consequence of \eqref{2ndcase2}. 
\smallskip

\noindent The differential inequality satisfied by the endomorphism $\eta_t$ implies that we have 
\begin{equation}\label{rev31}
\frac{1}{\delta_i}\int_{\wh X}\left\langle \Phi(u_i/\delta_i)(D'u_i), D'u_i\right\rangle dV_i+ \int_{\wh X} \Tr \big(u_i\Lambda_{i}\Theta(E, h_E)\big)dV_i\leq 0
\end{equation}
which shows that for each compact subset $K\subset \wh X\setminus D$ there exists an index $i_K$ such that 
\begin{equation}\label{rev32}
\int_{K}\left\langle \Psi(u_i)(D'u_i), D'u_i\right\rangle dV_i+ \int_{\wh X} \Tr \big(u_i\Lambda_{i}\Theta(E, h_E)\big)dV_i\leq 0
\end{equation}
for all $i\geq i_K$. The inequality \eqref{rev32} follows from the estimate \eqref{2ndcase2}, combined with the explicit expression of  
$\Phi$ and the properties of $\Psi$.
\medskip

\noindent Next, the absolute value of the eigenvalues of $u_i|_K$ are uniformly bounded as soon as $i\geq i_K$, and it follows from \eqref{rev32} that 
\begin{equation}
\int_{K}|D'u_i|^2dV_i\leq C
\end{equation}
for all $i\geq i_K$. By a diagonal procedure we can extract a sequence from $(u_i)_{i\geq 1}$ which converges weakly to an endomorphism 
$u_\infty$ in $W^{1, 2}$ on compact subsets contained in the complement of $D$. Moreover, by Rellich embedding theorem we can assume that 
\begin{equation}\label{rev34}
\Vert u_i-u_{\infty}\Vert_{L^2(K)}\to 0
\end{equation}
as $i\to \infty$, for any compact subset $K\subset X\setminus D$. 
\medskip

\noindent We now show that \eqref{rev34} implies that the limit $u_{\infty}$ is not identically zero. To this end, first note that by the normalisation in the second case we get 
\begin{equation}
\int_{\wh X}\delta_i \log {\Tr S_i}dV_i \geq  1- \int_{\wh X}\delta_i |\rho_i|dV_i 
\end{equation}

\noindent Since the trace of $s_i$ is equal to zero, we get
\begin{equation}\label{rev21}
\delta_i\log{\Tr(S_i)}\leq \delta_i\log r+ r|u_i|+ \delta_i |\rho_i|
\end{equation}
and therefore
\begin{equation}\label{rev10}
r\int_{\wh X}|u_i|dV_i \geq 1-\delta_i\log r- 2\int_{\wh X}\delta_i |\rho_i|dV_i.
\end{equation}
Assume that the weak limit (in $H^1$) of 
$\displaystyle (u_i)_{i\geq 1}$ is identically zero. We have a uniform bound for the $L^\infty$ norm of 
$\delta_i\rho_i$ on compact sets, which together with the equation defining $\rho_i$ shows that $(\delta_i\rho_i)$ converges to a smooth limit 
say $\rho_\infty$. Then we have
\begin{equation}\label{rev111}
u_\infty= -\frac{\rho_\infty}{r}\otimes \Id_E+ s_\infty
\end{equation}
a.e. on $\wh X$. If the LHS of \eqref{rev111} is zero, then both $\rho_\infty$ and $s_\infty$ must be equal to zero, as we see
by taking the trace in \eqref{rev111}. Combined with \eqref{new38}, this contradicts \eqref{rev10} as $i\to\infty$.
\end{proof} 

\medskip

\noindent Lemma \ref{weak} implies that:
\smallskip

\noindent $\bullet$ the eigenvalues of $u_\infty$ are constant a.e. on $\wh X$, and they are not all equal (here the trace of $u_\infty$ is not zero, but this is the case for the average of the trace showing that at least two of the eigenvalues of this endomorphism must be different). 
\smallskip

\noindent $\bullet$ some appropriate eigenspace of $u_\infty$ defines a destabilising subsheaf of $E$.
\medskip

\noindent The two bullets above are established in \cite{Sim88} and we will not reproduce the proof here. The fact that we are working with 
$\pi^\star\omega_X$ rather than with a genuine Kähler metric is not relevant. 
\medskip

\begin{proof}
\noindent Theorem \ref{HE} is an immediate consequence of Theorem 
\ref{Est, I}. The metric $h_\cF$ is obtained by push-forward of the limit metric combined with\eqref{mt5}.
\end{proof}
\medskip

\noindent Theorem \ref{HE} implies the next statement.

\begin{cor}\label{C0}
Assume that the hypothesis of Theorem \ref{HE} are satisfied. Then there exists positive constants $C> 0, N> 0$ such that the inequality
\begin{equation}\label{ross82}
\Tr(H)\leq C, \qquad H\geq C^{-1}|s_D|^N\Id_E
\end{equation}
holds at each point of $\wh X\setminus D$. In particular, it follows that given any local holomorphic section $\tau\in H^0(V, \cF|_V)$ defined on an open subset $V\subset X$, for any $V'\Subset V$ there exists a constant $C(\tau, V')> 0$ such that we have
\begin{equation}\label{ross83}
\sup_{V_0'}(|\tau|_{h_\cF})\leq C(\tau, V'),
\end{equation}
where $V_0':= V'\cap X_0$.
\end{cor}
\begin{proof}
The inequality \eqref{ross82} is part of Theorem \ref{HE}. Combined with Remark \ref{comparison}, this clearly gives \eqref{ross83}.
\end{proof}
\medskip

\section{Applications: a few intermediate results}\label{InterI}

\smallskip 

\noindent As application of Theorem \ref{HE} we have the next statement, showing that the lack of precision induced by the 
presence of $\phi_Z$ in Theorem \ref{HE} is immaterial in the context of orbifold singularities. In the next statement we assume that 
$B\subset \CC^n$ is the unit ball and $G$ is a finite group acting on $B$ holomorphically. We assume that the action is free in the complement of a set $W$ of codimension at least two and let $V:= B/G$ be the quotient.

\noindent Assume the the restriction $\cF|_V$ is a $\Q$-vector bundle, so that there exists a vector bundle $\sE\to B$ compatible with the action of $G$ and whose $G$-invariant sections correspond to the sections of $\cF|_V$. Now, since 
$$\sE= \big(\pi^\star \cF|_V\big)^{\star\star}$$
where $p:B\to V$ is our ramified cover, the pull-back of the HE metric $h_\cF$ defines a metric on $\sE|_{B\setminus W}$.
More explicitly, let $V_0\subset V\cap X_{\rm reg}$ be the open subset such that $\cF|_{V_0}$ is a vector bundle. Then we have 
\begin{equation}\label{ross200}\sE|_{B\setminus W}= p^\star (\cF|_{V_0})
\end{equation}
so that $\sE$ can be seen as extension of the vector bundle $\displaystyle p^\star (\cF|_{V_0})$ across $W$. 

\noindent Therefore the metric $h_{\sE}$ compares with a fixed, arbitrary Hermitian metric $h_0$ on $\sE$ as follows
\begin{equation}\label{ross201}
C^{-1}\exp(N\phi_Z\circ \pi)h_0\leq h_{\sE}\leq Ch_0
\end{equation}
since this is the case for the pull-back of $h_{\cF, 0}$ (by the definition of this metric), and then we use Theorem \ref{HE}. One can argue as follows. In the first place we assume that $V$ is sufficiently small, so that the restriction $\cF|_V$ is given by the exact sequence
\[\mathcal O_V^p\to\mathcal O_V^q\to \cF|_V\to 0.\]
Given the equality \eqref{ross200}, the pull-back of this sequence to $U$ gives
\begin{equation}\label{ross500}\mathcal O_B^p\to\mathcal O_B^q\to \sE|_B\to 0\end{equation} 
which is moreover exact when restricted to the complement $B\setminus W$ of $W$ in $B$. The pull-back of $h_{\cF, 0}$ on
$\displaystyle \sE|_{B\setminus W}$
is the quotient metric induced by \eqref{ross500} above (as explained in Section 1). The important point is that 
$\codim_B (W)\geq 2$ is at least two, and therefore the holomorphic maps defining the morphisms in \eqref{ross500} locally near a 
point of $W$ extend. Then \eqref{ross201} follows.  
\medskip

\noindent Moreover, we denote by 
\begin{equation}\label{ross110}\omega_B:= \pi^\star\omega_X
\end{equation}
the inverse image of the metric $\omega_X$ of $X$.
\medskip

\noindent In this context we have the following result.
\begin{thm}\label{orbi, I} Let $\tau$ be any holomorphic section of $\sE$. Then the equality 
 \begin{equation}\label{ross111} \Delta''_{\omega_B}|\tau|_{h_\sE}^2= |D'\tau|^2+ C|\tau|^2_{h_\sE}
\end{equation}
holds pointwise on $B\setminus W$ and in the sense of distributions on $B$, where $C$ is a constant independent of $\tau$.
\end{thm}

\noindent Before presenting the proof, a few words about this statement: by \emph{in the sense of distributions} we mean that for any smooth function 
$\theta$ with compact support on $B$ the equality
\[n\int_B|\tau|_{h_\sE}^2\sqrt{-1}\ddbar \theta\wedge \omega_B^{n-1}= \int_B\theta(|D'\tau|^2+ C|\tau|^2_{h_\sE})\omega_B^n\]
holds true.

\begin{proof}
\noindent By the usual Poincaré-Lelong formula we have
\begin{equation}\label{conc3}
\Delta''_{\omega_B}|\tau|_{h_\sE}^2=  |D'\tau|^2 +\langle\Lambda_{\omega_B}\sqrt{-1}\Theta(\sE, h_\sE)\tau, \tau\rangle
\end{equation}
at each point of $B\setminus W$. 

\noindent Thanks to the HE condition we get
\begin{equation}\label{conc4}
\langle\Lambda_{\omega_B}\sqrt{-1}\Theta(\sE, h_\sE)\tau, \tau\rangle= C|\tau|^2_{h_\sE}
\end{equation}
\noindent for some constant $C$, so the first part of Theorem \ref{orbi, I} is proved; we will show next that the inequality \eqref{conc3} holds true on $B$ in the sense of distributions.
\smallskip

\noindent We assume that $W= (f_1=0,\dots, f_r= 0)$ where $f_i$ are
holomorphic. Let $p: Y\to B$ be a composition of blow-up maps with smooth centre, such that the inverse image of the ideal $(f_1,\dots, f_r)$ is principal.

\noindent Let $(D_i)_i$ be the set of exceptional divisors of the map $p$; they are assumed to have transverse intersections at each point. Let
$(w^1=0)$ be the local equation of some of the $D_i$. Then we claim that the
$(n,n)$--form \begin{equation}\label{conc6}\sqrt{-1}dw^1\wedge d\ol w^{1}\wedge p^\star\omega^{n-1}_B\end{equation}
equals $\O(|w^1|)$.

\noindent If $p$ is very simple, e.g consists in one single blow-up, then the local equation of this map reads as
\begin{equation}\label{conc5}
w\to (w^1, w^1w^2, \dots,w^1w^k, w^{k+1},\dots, w^n)
\end{equation}
for a certain $2\leq k\leq n$. Then we see that the form in \eqref{conc6} is divisible by $|w^1|^{2k-2}$.

\noindent The general case follows from a less precise --but more direct-- argument: we observe that we have 
\begin{equation}\label{part1}
\int_{D_i}p^\star\omega^{n-1}_B= 0
\end{equation}
and on the other hand the wedge product in \eqref{conc6} is smooth, hence it must be equal to $\O(|w^1|)$.

\smallskip

\noindent Let $\phi$ be a positive test function on $B$ and $\Xi_\ep$ be the
usual truncation function, that is to say $\displaystyle \rho_\ep\Big(\log\log \frac{1}{\sum|f_i|^2}\Big)$. We have to show that
\begin{equation}\label{conc7}
\int_B|\tau|^2_{h_\sE}\phi\Delta''(\Xi_\ep)\omega_B^{n}\to 0, \qquad \int_B|\tau|^2_{h_\sE}\langle d\phi, d\Xi_\ep\rangle\omega_B^{n}\to 0
\end{equation}
as $\ep\to 0$. 

\noindent To this end, we recall that the estimate
\begin{equation}\label{conc8}
\sup_{B\setminus W}|\tau|^2_{h_\sE}\leq C(\tau)
\end{equation}
holds by \eqref{ross201}. The terms involving derivatives of the truncation function $\Xi_\ep$ can be bounded by integrals of the following type
\begin{equation}\label{conc9}
 \int_B\phi\rho_\ep'\Big(\log\log \frac{1}{\sum|f_i|^2}\Big)\frac{\sqrt{-1}\partial f_i\wedge \ol{\partial f_i}}{\sum |f_i|^2}\wedge {\omega_U^{n-1}}
\end{equation}
as consequence of the inequality \eqref{conc8}. But \eqref{conc9} converges to zero as $\ep\to 0$, as we see by blowing up and using the claim \eqref{conc6}. 
\end{proof}

\medskip

\noindent The same method of proof gives the following result.
\begin{thm}\label{orbi, II} Let $\mu$ be any holomorphic section of $\sE^\star$. Then we have 
 \begin{equation}\label{ross202} \Delta''_{\omega_B}\log|\mu|_{h_\sE^\star}^2\geq -C
\end{equation}
holds pointwise on $B\setminus W$ and in the sense of distributions on $B$, where $C> 0$ is a constant independent of $\mu$.
\end{thm}

\begin{proof}
Indeed, given the inequality \eqref{ross201}, we have 
\begin{equation}\label{ross203}
\sup_{B\setminus W}\big|\log|\mu|^2_{h_\sE^\star}\big|\leq C(\mu)\log\frac{1}{\sum |f_i|^2}
\end{equation}
and this is mild enough in order to have the integrals \eqref{conc7} converging to zero.
\end{proof}
\medskip

\noindent As consequence of the previous Theorem \ref{orbi, I} and Theorem \ref{orbi, II} we have the following statement.
\begin{cor}\label{order0}
There exists a positive constant $C>0$ such that the inequalities
\begin{equation}\label{ross204}
C^{-1}h_0\leq h_{\sE}\leq Ch_0
\end{equation}
hold on $B\setminus W$.
\end{cor}

\begin{proof}
The Laplace inequality established in Theorem \ref{orbi, II} combines with Moser iteration procedure as in \cite{CP22} and shows that we have
\begin{equation}\label{ross205}
\sup_{B\setminus W}|\mu|^2_{h_\sE^\star}\leq C(\mu),
\end{equation} 
and the statement follows.
\end{proof}
\medskip

\begin{remark}\label{orb} Corollary \ref{order0} applies for any metric $h_\sE$ which is HE with respect to a Kähler metric $\omega$ on $X$, 
provided that $p^\star\omega$ verifies Sobolev inequality.
\end{remark}

\noindent Moreover, we can analyse the metric $h_\sE$ a bit further, as follows. 
\begin{cor}\label{order1}
We assume that the hypothesis in Theorem \ref{orbi, I} are fulfilled. Let $C> 0$ be a positive constant, and consider the class of smooth positive functions $\psi$ such that the inequality
$$\int_B(\psi+ |\Delta''_{\omega_B} \psi|)\omega_B^{n}\leq C$$
holds. Then we have $\displaystyle \int_B\psi |D'\tau|_{h_\sE}^2\omega_B^{n}\leq C_0C\sup_{B\setminus W}|\tau|^2$, where $C_0$ is a constant only depending on $B$. It follows that there exists a positive $\ep_0> 0$ such that 
\begin{equation}
\int_{B\setminus W}\frac{\sum_{\alpha, \beta}|A^\alpha_\beta|^2}{(\sum|f_i|^2)^{\ep_0}}\omega_B^{n}< \infty,
\end{equation}
where $(A^\alpha_\beta)$ are the coefficients of the Chern connection of $h_\sE$ with respect to the basis provided by a fixed trivialisation of $E$.
\end{cor}

\begin{proof} Indeed, we use the equality
\begin{equation}\label{coro1}
\Delta''_{\omega_B} |\tau|^2=  |D'\tau|^2- C|\tau|^2
\end{equation}
established in Theorem \ref{orbi, I}. 
We multiply by $\psi$, integrate and the conclusion follows. Note that the exponent $\ep_0$ depends on the ideal generated by the 
functions $(f_i)$.\end{proof}
\medskip

\begin{remark} For example, assume that $B\subset \CC^3$ is the unit ball, and that the singular set $W$ is given by the vanishing of the first two coordinates.
Consider the function $\psi$ given by the expression
\begin{equation}\label{3folds1}
\psi(z):= \frac{1}{(|z_1|^2+ |z_2|^2)^{\ep_0}},
\end{equation}
where $\ep_0> 0$ is a positive constant we want to determine, so that the
differential inequality in Corollary \ref{order1} is satisfied.

\noindent To this end, a quick calculation shows that we have
\begin{equation}\label{3folds2}
|\Delta\psi|\omega_B^{n}\leq \frac{C}{(|z_1|^2+ |z_2|^2)^{1+ \ep_0}}d\lambda  
\end{equation}
where $C$ is a positive constant --actually, this constant in
\eqref{3folds2} can be computed as soon as we fix $C_1>0$ such that
$\omega_B\leq C_1\omega_{\rm euc}$.

\noindent It is therefore clear that we can take any $\ep_0< 1$, so in conclusion we have
\begin{equation}\label{3folds3}
\int_{B\setminus W}\frac{\sum_{\alpha, \beta}|A^\alpha_\beta|^2}{(\sum|z_i|^2)^{2-\delta}}\omega_B^{3}< \infty  
\end{equation}
for any $\delta> 0$. \end{remark}
\medskip

\subsection{Chern classes inequality} In this subsection we prove Theorem \ref{Chern1}. 
\smallskip

\noindent $\bullet$ \emph{Proof of \rm (1)} It is sufficient to show that there exists a positive constant $C> 0$ such that 
we have
\begin{equation}\label{Chern11}
|\Delta(\cF, h_\cF)|_{\omega_X}\leq C|\Theta(\cF, h_\cF)|_{\omega_X, h_\cF}^2
\end{equation}
pointwise on $X_0$. This is done as follows.

\noindent Let $x_0\in X_0$ be an arbitrary point. Then we have coordinates $(z_i)$ centred at $x_0$ which moreover are geodesic for the metric $\omega_{X}$
and a normal frame $e_\alpha$ for $(\cF, h_{\cF})$. We write 
\begin{equation}\label{Chern12} 
\Theta(\cF, h)_{x_0}= \sum \theta^\alpha_{\beta p\ol q}dz^p\wedge dz^{\ol q}\otimes e_\alpha\otimes e_\beta^\star 
\end{equation}
Then we have 
\begin{equation}\label{Chern13} 
|\Theta(\cF, h)|_{x_0}^2dV_{\omega_{X}}= \sum |\theta^\alpha_{\beta p\ol q}|^2d\lambda
\end{equation}
at the point $x_0$. Then the inequality \eqref{Chern11} follows, since the expression of $\Delta(\cF, h_\cF)$ at $x_0$ 
is quadratic with respect to the 
curvature coefficients.  
\smallskip

\noindent $\bullet$ \emph{Proof of \rm (2)} From this point on we assume that $\dim(X)= 3$.
Let $\alpha$ be a smooth $(1, 0)$-form with support in $V$. It would be enough to prove that the following holds
\begin{equation}\label{3folds4}
\lim_{\ep\to 0}\int_{V}\Delta(\cF, h_\cF)\wedge \dbar\chi_\ep\wedge \alpha= 0.
\end{equation}
By pulling back to $B$ via the map $p$, the main assignment in \eqref{3folds4} is the following
\begin{equation}\label{3folds5}
\lim_{\ep\to 0}\int_{B}\Tr\big(\Theta(\sE, h_\sE)\wedge \Theta(\sE, h_\sE)\big)\wedge p^\star(\dbar\chi_\ep\wedge \alpha)=0,
\end{equation}
Locally on $B$ we have the equality $\displaystyle \Theta(\sE, h_\sE)= \dbar A$, where $A= (A^p_q)$ is the matrix of $(1,0)$-forms above,
so \eqref{3folds5} is equivalent to 
\begin{equation}\label{3folds6}
\lim_{\ep\to 0}\int_{B}\Tr\big(A\wedge \Theta(\sE, h_\sE)\big)\wedge p^\star(\dbar\chi_\ep\wedge \dbar \alpha)=0.
\end{equation} 
where we use the notation $h_s$ for the inverse image of the metric $h_\cF|_V$.

\noindent By using Cauchy-Schwarz inequality, it would be sufficient to prove that 
\begin{equation}\label{3folds7}
\lim_{\ep\to 0}\int_{B}|A|^2_{\omega_B}\cdot |p^\star(\dbar\chi_\ep)|^2_{\omega_B}\omega_B^3<\infty.
\end{equation} 
since the $L^2$ norm of the curvature tensor $\Theta(\sE, h_\sE)$ with respect to the inverse image metric $\omega_B$ is finite.
\smallskip

\noindent Let $x_0\in V\subset X_{\rm sing}$ be a point belonging to the singular set of $X$, such that 
the composition of the ramified cover $p:B\to V$ composed with the local embedding of $(X, x_0)$ in $\CC^4$ can be written as follows 
\begin{equation}\label{3folds8}
(z, u, v)\to \big(z, f_1(u, v), f_2(u, v),f_3(u, v)\big)
\end{equation}
where the $f_i$'s are local holomorphic functions. For example, in the case of a singularity of type $A_n$ we have 
$$f_1(u, v)= u^n,\qquad f_2(u, v)= v^n, \qquad f_3(u, v)= uv.$$
\smallskip 

\noindent Consider next the truncation function
\begin{equation}\label{3folds9}
\chi_\ep:= \Xi_\ep\big(\log\log(1/\sum_{i=2}^{4} |Z_i|^2)\big), \qquad \omega_X\simeq \sqrt{-1}\sum_{i=1}^{4}dZ_i\wedge d\ol Z_i
\end{equation}
where we denote by $\simeq$ the fact that the two metrics are quasi-isometric. 

\noindent A rough estimate gives
\begin{equation}\label{3folds10}
|\dbar\chi_\ep|_{\omega_X}\leq \frac{-\Xi'_\ep(\wt Z)}{(|\wt Z|^2)^{1/2}\log(1/|\wt Z|^2}=: \psi_\ep^{1/2}(\wt Z)
\end{equation}
where $\wt Z:= (Z_2, Z_3, Z_4)$.

\noindent We have to evaluate the expression
\begin{equation}\label{3folds11}
\int_B|\Delta_{\omega_B}(\psi_\ep\circ p) |\omega_B^3
\end{equation}
and since the following formula holds
\begin{equation}\label{3folds14}
\Delta_{\omega_B}(\psi_\ep\circ p)= (\Delta_{\omega_X}\psi_\ep)\circ p
\end{equation}
it would be sufficient to show that the family of integrals
\begin{equation}\label{3folds13}
I_\ep:= \int_V\theta(Z) |\Delta_{\omega_X}\psi_\ep|dV_{\omega_X}
\end{equation}
converges to zero as $\ep\to 0$.
\smallskip

\noindent In the expression \eqref{3folds13}, the most singular term is equal to 
\begin{equation}\label{3folds12}
J_\ep:= -\int_{V}\frac{\theta(Z) \Xi'_\ep(\wt Z)}{|\wt Z|^{4}\log^2(1/|\wt Z|^2)}\omega_{\rm euc}^3
\end{equation}
as a quick calculation shows it.

\noindent We assume that $V= (f= 0)\subset (\CC^4, 0)$, that is to say locally near the point $x_0=0$, the 
space $X$ is given by the vanishing of the holomorphic function $f$. Then the next equality holds
\begin{equation}\label{3folds15}
J_\ep= -\int_{(\CC^4, 0)}\frac{\theta(Z) \Xi'_\ep(\wt Z)}{|\wt Z|^{4}\log^2(1/|\wt Z|^2)}\omega_{\rm euc}^3\wedge \ddc\log |f|^2
\end{equation}
by the Poincar\'e-Lelong formula, and integration by parts shows that we have
\begin{equation}\label{3folds16}
J_\ep= \int_{(\CC^4, 0)}\log 1/|f|^2\ddc\Big(\frac{\theta(Z) \Xi'_\ep(\wt Z)}{|\wt Z|^{4}\log^2(1/|\wt Z|^2)}\Big)\wedge\omega_{\rm euc}^3. 
\end{equation}
\smallskip

\noindent The most "dangerous" term would be 
\begin{equation}\label{3folds17}
\int_{(\CC^4, 0)}\frac{\log 1/|f|^2}{\log^2(1/|\wt Z|^2)}\theta(Z) \Xi'_\ep(\wt Z) \ddc\Big(\frac{1}{|\wt Z|^{4}}\Big)\wedge\omega_{\rm euc}^3
\end{equation}
but very luckily, it is equal to zero since we integrate in the complement of $\wt Z=0$ the Laplacian of the Newton potential in $\CC^3$. 

\noindent The other terms are bounded by integrals of the following type
\begin{equation}\label{3folds18}
\int_{(\Supp \Xi_\ep', 0)}\frac{\log 1/|f|^2}{|\wt Z|^{6}\log^3(1/|\wt Z|^2)}\omega_{\rm euc}^4
\end{equation}
and the integral in \eqref{3folds18} converges to zero as $\ep\to 0$.
\smallskip

\noindent $\bullet$ \emph{Proof of \rm (3)} Since the dimension of $X$ is equal to 3, the set $W$ consists in a 
finite number of points. We denote by $\Xi_{\ep, W}$ the truncation function corresponding to $W$ defined on the
Euclidean space and restricted to $X$. 

\noindent The equality 
\begin{equation}\label{3folds30}
\int_{X\setminus Z}\Delta(\cF, h_{\cF})\wedge (\omega_X+ \ddc\varphi)= \int_{X\setminus Z}\Xi_{\ep, W}\Delta(\cF, h_{\cF})\wedge (\omega_X+ \ddc\varphi)
\end{equation}
holds as soon as $\ep\ll 1$ since $\omega_X+ \ddc\varphi$ is equal to zero locally near $W$. 

\noindent By the HE condition we have
\begin{equation}\label{3folds31}
\int_{X\setminus Z}\Xi_{\ep, W}\Delta(\cF, h_{\cF})\wedge \omega_X\geq 0
\end{equation}
and moreover, the point (2) above shows that the integration by parts 
\begin{equation}\label{3folds32}
\int_{X\setminus Z}\Xi_{\ep, W}\Delta(\cF, h_{\cF})\wedge \ddc\varphi= -\sqrt{-1}\int_{X\setminus Z}\Delta(\cF, h_{\cF})\wedge 
\dbar \Xi_{\ep, W}\wedge \partial \varphi
\end{equation}
is justified. 

\noindent To finish the argument, remark that there exists a positive constant $C> 0$ such that  
\begin{equation}\label{3folds33}
\sup_{X\setminus Z}|\dbar \Xi_{\ep, W}\wedge \partial \varphi|_{\omega_X}\leq  C
\end{equation}
because $-\varphi$ is a K\"ahler potential near $W$, and we can choose it to vanish up to order two at each point of $W$
(here we are referring to the coordinates in the Euclidean space containing the points of $W$).

\noindent All in all, it follows that the RHS of \eqref{3folds32} is bounded by the integral
\begin{equation}\label{3folds34}
C\int_{\Supp \Xi_{\ep, W}'}|\Theta(\cF, h_{\cF})|^2dV_{\omega_X}
\end{equation} 
which converges to zero since the curvature is $L^2$. Theorem \ref{Chern1} is proved.\qed
\medskip


\section{Orbifold-type K\"ahler metrics on klt spaces}\label{MAM}

\smallskip

\noindent Let $X$ be a normal K\"ahler space with at most klt singularities, endowed with a Kähler metric $\omega_X$. It is known (by the work of \cite{GKKP11}) that there exists $Z\subset X$ such that $\codim_X(Z)\geq 3$ and such that the complement $X\setminus Z$ only has quotient singularities. For example, if $X$ has dimension 3, it follows that $Z$ consists in at most a finite number of points $(p_i)_{i\in I}$. 
\smallskip

\noindent The main results of the current section are as follows. Given any open subset $U\subset X$ containing $Z$ above
we start by  
constructing a
K\"ahler metric $\omega_{\rm orb}$ on $X$ which accounts for the quotient singularities when restricted to the complement of $U$. By using Monge-Ampère 
theory we show that there exists a Kähler approximation of the form $\pi^\star\omega_{\rm orb}+ t\omega_{\wh X}$ whose $L^p(\wh X, \omega_{\wh X})$--norm of the determinant is uniformly bounded (with respect to $t$ and the approximation parameter) for some $p> 1$.  
This result, combined with \cite{Phong} will show that a uniform mean-value inequality holds true for the family of the approximating metrics.
Finally, this is used to establish Theorem \ref{HEquot}.

\subsection{Construction of $\omega_{\rm orb}$}\label{orbmetric} 
Let $Z\subset V_0'\Subset V_0''\Subset U$ be two open subsets of $X$ containing $Z$ and 
contained in a pre-assigned open subset $U\subset X$. We consider 
a finite cover 
$(V_i'\Subset V_i'')_{i=1,..., M}$ of $X\setminus \ol {V_0''}$ such that the following hold.
\begin{enumerate}
	\smallskip
	
	\item We have $\displaystyle X\setminus \ol {V_0''}\subset \cup_{i\geq 1} V_i '$ and none of the sets $V_i''$ are intersecting $V_0'$, for $i\geq 1$.
	\smallskip
	
	\item For each index $i\geq 1$ there exists a ball $U_i\subset \CC^n$ together with a finite group $\Gamma_i$ acting
	on $U_i$ such that $V_i''\simeq U_i/\Gamma_i$.
	\smallskip
	
	\item\label{forb} Let $\pi_i: U_i \rightarrow  U_i/\Gamma_i \simeq V_i ''$ be the cover. There exists a positive continuous function $w_i$ on $V_i''$ such that 
	$$w_i\leq 1, \qquad \pi_i^\star(\sqrt{-1}\ddbar w_i)\simeq \omega_{\rm euc}.$$
	Indeed such function is obtained by averaging the usual $\Vert z_i\Vert^2$ and dividing by an appropriate
	constant.
\end{enumerate}

\noindent By a gluing argument --i.e. by using cutoff functions-- we can construct a continuous function 
$\varphi: X\to [-\infty, 0[$ such that
\begin{equation}\label{new1}
	\sqrt{-1}\ddbar \varphi\geq -\frac{C}{2}\omega_X ,
\end{equation}
where $C$ is a constant. 

\noindent Actually \eqref{new1} follows directly from the gluing lemma, cf. \cite[Lemma 3.5]{Dem92}, as we briefly recall next. Let $\theta_i$ be a smooth function $\equiv 1$ on $V_i '$ and whose support is contained in $V_i ''$. We then consider 
$$\varphi := \log \big(\theta_0 ^2 + \sum_{i=1}^M \theta_i ^2(1+ w_i)\big),$$
where $w_i$ is the function in the item (\ref{forb}) above.

\noindent Given that on overlapping sets $V_i''\cap V_j''$ we have 
\begin{equation}\label{new40}
	C_{ij}^{-1}<\frac{1+ w_i}{1+ w_j}< C_{ij}
\end{equation}
for some positive constants $C_{ij}> 0$ (and with the convention that $w_0=0$ if some of the $i, j$ above are equal to zero).  Moreover, we can find a constant $C$ such that 
$$ \sqrt{-1} (\theta_j \partial\dbar\theta_j - \partial \theta_j \wedge \dbar \theta_j) \geq -C \omega_X \qquad\text{on } V_j '' . $$
Thanks to \cite[Lemma 3.5]{Dem92}, it follows that 
\eqref{new1} holds true.
\bigskip

\noindent We introduce then the metric 
\begin{equation}\label{new2}
	\omega_{\rm orb}:= C\omega_X+ \sqrt{-1}\ddbar \varphi  
\end{equation}
and summarise its properties as follows.  



\begin{enumerate}
	\smallskip
	
	
	\item For each uniformization map $\pi: \widetilde{V} \to V= \widetilde{V} /\Gamma$ such that $V\subset (X\setminus U)$ the pull-back $\displaystyle \pi^\star(\omega_{\rm orb})$ is quasi-isometric with the Euclidean metric. 
	\smallskip
	
	\item We have $\displaystyle \omega_{\rm orb}= C\omega_X $ when restricted to an open subset $\Omega$ containing $Z$. Here $\Omega\Subset U$. Moreover, the inequality $\displaystyle \omega_{\rm orb}\geq C\omega_X$ holds globally on $X$, where $C> 0$ is some positive constant (not necessarily equal to the previous one). 
\end{enumerate}
\medskip

\noindent We next discuss an important property of the metric $\omega_{\rm orb}$. When restricted to $X_{\rm reg}$, we can write its determinant as follows
\begin{equation}\label{new45}
	\omega_{\rm orb}^n= \Psi\omega_X^n
\end{equation}
for some function $\Psi> 0$. Our next statement shows that $\Psi$ is $L^p$-integrable on $(X_{\rm reg}, \omega_X)$ for some $p> 1$.
\smallskip

\begin{lemme}\label{lp} There exists $p>1$ such that the integral
	\begin{equation}\label{new46}
		\int_{X_{\rm reg}}\Psi^p\omega_X^n<\infty
	\end{equation}
	is convergent. 
\end{lemme}
\begin{proof} Let $(V_\alpha)_{\alpha=1,\dots, M}$ be a finite cover of 
	$X$ such that 
	\begin{equation}
		\frac{1}{2}\Omega\subset \cup_{\alpha=1}^{M'} V_\alpha\subset \Omega
	\end{equation}
	and such that 
	\begin{equation}
		\cup_{\alpha=1+ M'}^{M} V_\alpha \subset X\setminus \frac{1}{4}\Omega
	\end{equation}
	together with $V_\alpha\simeq U_\alpha/\Gamma$ for $\alpha= 1+ M', \dots , M$. 
	\smallskip
	
	\noindent We analyze the convergence of the integral \eqref{new46} by restricting to each of the open sets $(V_\alpha)$ above.
	In the first place, if $\alpha= 1,\dots, M'$ we have 
	\begin{equation}\label{new47}
		\Psi|_{V_\alpha}\leq C
	\end{equation}
	by the property (2) above and there is nothing to prove (since the volume of $X_{\rm reg}$ is finite). If $\alpha\geq 1+ M'$, then we use the 
	local uniformisations $p_\alpha : U_\alpha \rightarrow V_\alpha \simeq U_\alpha/\Gamma$. Since $p_\alpha ^* \varphi$ is smooth, so we see
	that 
	\begin{equation}\label{new48}
		\Psi|_{V_\alpha}\circ p_\alpha\leq \frac{C}{|\Jac(p_\alpha)|^2}
	\end{equation}
	for some constant $C$,  and therefore 
	\begin{equation}\label{new49}
		\int_{V_\alpha}\Psi^{1+\ep_0}\omega_X^n|_{V_\alpha}\leq C \int_{U_\alpha} \frac{1}{|\Jac(p_\alpha)|^{2\ep_0}}d\lambda
	\end{equation}
	and this last quantity is clearly convergent, as soon as $0<\ep_0\ll 1$.  
\end{proof}
\medskip

\noindent Similar arguments allow us to deduce the following. Assume that $t\in [0,1]$ is a positive real number and define the function $\Psi_t$ by the following equality
\begin{equation}\label{new50}
	(\pi^\star \omega_{\rm orb}+ t\omega_{\wh X})^n= \Psi_t\omega_{\wh X}^n.
\end{equation}
Then we claim that there exists a positive $\ep_0$ such that we have 
\begin{equation}\label{new51}
	\int_{\wh X}\Psi_t^{1+\ep_0}\omega_{\wh X}^n\leq C< \infty
\end{equation} 
for some constant $C$ which is \emph{independent of t}. 

\noindent Indeed, this can be verified as in the previous lemma: for each index $\alpha\geq M'+1$, let $\pi_\alpha: \wh U_\alpha\to U_\alpha$ be a modification of $U_\alpha$ such that there exists a generically finite morphism $\wh p_\alpha: \wh U_\alpha\to \pi^{-1}(V_\alpha)$ with the property that
$$\pi\circ \wh p_\alpha= p_\alpha \circ\pi_\alpha$$
holds true (we take any desingularisation of the fibered product $\displaystyle U_\alpha\times_{V_\alpha}\pi^{-1}(V_\alpha)$.)

\noindent Let $\Psi_k$ be the function defined by the equality 
$$\pi^\star \omega_{\rm orb}^k\wedge \omega_{\wh X}^{n-k}= \Psi_k\omega_{\wh X}^{n}$$
and then we remark that we have 
$$\Psi|_{k}\circ \wh p_\alpha\leq \frac{C}{|\Jac(\wh p_\alpha)|^2},$$
that is to say, the analogue of \eqref{new48} holds true. Then we argue as in \eqref{new49} and \eqref{new51} follows. \qed

\medskip

\subsection{Few results from Monge-Amp\`ere theory}

\noindent Next we recall an important estimate from MA theory. Let $\pi: \wh X\rightarrow X$ be a desingularization and we fix a smooth Kähler metric $\omega_{\wh X}$ on $\wh X$. We consider a real number $p> 1$, a positive constant $C_0> 0$ and a 
positive function $\Psi\geq 0$ such that 
\begin{equation}\label{ma1}
	\int_{\wh X}\Psi^p\omega_{\wh X}^n< C_0.
\end{equation}
Associated with such function $\Psi$ we consider the Monge-Ampère equation
\begin{equation}\label{new52}
	(\pi^\star \omega_{\rm orb}+ t\omega_{\wh X}+ \sqrt{-1}\ddbar \phi)^n= \Psi\omega_{\wh X}^n, \qquad \sup_X \phi=0.
\end{equation}
which by the work of \cite{Kol98, EGZ09, TZ06} admits a unique solution. 
Moreover, the following fundamental estimate
holds 
\begin{equation}\label{new53}
	\sup_X|\phi|\leq C= C(\omega_X, \omega_{\wh X},C_0, p)
\end{equation} 
i.e. the $L^\infty$ norm of the solution of \eqref{new52} is bounded by a constant which is independent of $t$ and the explicit shape of 
$\Psi$, as long as the $L^p$ norm is satisfied. In addition, if $\Psi$ is smooth in the complement of $D$, then so is $\phi$.
\medskip

\noindent In order to apply the arguments in the smooth case (sections 3-4)
we have to construct a smooth approximation of the metric 
$\omega_t= \pi^* \omega_{\rm orb} +t \omega_{\wh X}$ in such a way that the entropy function is uniformly bounded. This is done as follows.
\medskip

\noindent Consider the Monge-Ampère equation
\begin{equation}\label{ma12}(\pi^* \omega_{\rm orb} +t \omega_{\wh X} + \ddc \phi_{t,\ep})^n = (\Psi_t \star \tau_\ep) \omega_{\wh X} ^n,
\end{equation}
whose solution is normalized by the condition $\max_{\wh X}(\phi_{t,\ep}+ \varphi)= 0$. In \eqref{ma12} we denote by $(\tau_\ep)_{\ep> 0}$ a family of 
regularizing kernels on $\wh X$. 
\smallskip

\noindent Note that $\displaystyle \omega_{\rm orb} =\omega_X + \ddc \varphi$ for the function $\varphi$ defined in section \ref{orbmetric}, where $\omega_X$ is the reference Kähler metric on $X$. 
Then we can rewrite \eqref{ma12} as follows
\begin{equation}\label{ma13}(\pi^* \omega_{X} +t \omega_{\wh X} + \ddc (\phi_{t,\ep} +\varphi) )^n = (\Psi_t \star \tau_\ep) \omega_{\wh X} ^n 
\end{equation}
and now the form $\pi^* \omega_{X} + t\omega_{\wh X}$ is smooth and definite positive for each $t> 0$. Moreover, there exists strictly positive constants $C, \delta_0$ such that we have 
\begin{equation}\label{ma14}\int_{\wh X}(\Psi_t \star \tau_\ep)^{1+ \delta_0}\omega_{\wh X}^n < C
\end{equation}
for all $t$ and $\ep$.
Thus we have 
\begin{equation}\label{ma15}\sup_{\wh X}|\psi_{\eta}|\leq C, \qquad \sup_{\wh X}|\phi_{\eta}  |\leq C
\end{equation}
where $\psi_{\eta}:= \phi_{t, \ep} + \varphi$, $\phi_\eta := \phi_{t,\ep}$ for a constant $C> 0$ independent of $\eta:= (t, \ep)$. 
\medskip

\noindent The following statement is the crucial step towards the construction of the Hermitian-Einstein metric with respect to 
$\omega_{\rm orb}$. 
\begin{cor}\label{rich}
For each couple $\eta= (t, \ep)$ we consider the metric 
	$$\omega_\eta := \pi^* \omega_{X} +t \omega_{\wh X} + \ddc \psi_\eta =  \pi^*\omega_{\rm orb}  +t \omega_{\wh X} + \ddc \phi_{\eta}.$$
	Let $\theta$ be a smooth $(1,1)$-form on $\wh X$ and let $f$ be a $C^2$-function on $\wh X$ such that 
	\begin{equation}\label{twistedlap}\Lambda_{\omega_\eta}\theta + \Delta_{\omega_\eta}'' f \geq c
	\end{equation}
on $\wh X$. 
\begin{enumerate}
\smallskip

\item[\rm (a)] 
There exists a constant $C > 0$ depends only on $\wh X, \omega_{\wh X}, \theta, c$ --in particular independent of $\eta = (t,\ep)$-- such that the inequality
	\begin{equation}\label{meanvalue}\sup_{\wh X}\left(C\log |s_D|^{2}+ f\right)\leq C(1+ \int_{\wh X}|f| \omega_\eta ^n )
	\end{equation}
	holds true.
\smallskip

\item[\rm (b)] In addition to \eqref{twistedlap}, we assume that $\theta\leq C_\theta\pi^* \omega_X$. Then there exists a constant $C$ depending on the quantities in (a) and $C_\theta$ such that 
the inequality
\begin{equation}\label{meanvalue, I}\sup_{\wh X} f\leq C(1+ \int_{\wh X}|f| \omega_\eta ^n )
	\end{equation}
is satisfied.   
\end{enumerate}

\end{cor}



\begin{proof} Given the arguments for the proof of Corollary \ref{mv}, all we still have to verify is that the Laplace inequality for $f$
with respect to $(\wh X, \omega_\eta)$ hold true. The slight trouble here is that we do not have a uniform lower bound for $\omega_\eta$ (this is the case for $\omega_t$ and it was very useful...).  
We therefore proceed as follows.
	
	\noindent The inequality
	\begin{equation}\label{ma17}\theta_L \leq C(\pi^* \omega_X +t \omega_{\wh X}+ \sqrt{-1}\ddbar \log|s_D|^{2\ep_0})
	\end{equation}
	holds true in the sense of currents for some positive constant $C$.
	Then 
	\begin{equation}\label{ma18}
		\theta_L - C\sqrt{-1}\ddbar   (\log|s_D|^{2\ep_0}- \phi_{\eta} - \varphi)\leq C(\pi^* \omega_{\rm orb} +t \omega_{\wh X} + \ddc \phi_{\eta}).
	\end{equation}
	Since we have $\omega_{\eta} = \pi^* \omega_{\rm orb} +t \omega_{\wh X} + \ddc \phi_{\eta}$,
	it follows that the inequality
	\begin{equation}\label{ma19}\Lambda_{\omega_{\eta}} (\theta_L - C\sqrt{-1}\ddbar ( \log|s_D|^{2\ep_0}  -\phi_{\eta} -\varphi)) \leq C\end{equation}
	is satisfied. Together with \eqref{twistedlap}, we obtain
	\begin{equation}\label{ma119}
		\Delta''_{\omega_\eta}  ( \log|s_D|^{2\ep_0}  -\phi_{\eta} -\varphi +f) \geq -C_1 \end{equation}
	\smallskip
	
	\noindent The inequality \eqref{ma14} implies that the entropy $\omega_{\eta}$ is uniformly bounded, we thus infer that the inequality 
	\begin{equation}\label{ma16}\sup_{\wh X}\left(C\log |s_D|^{2}+ f\right)\leq C(1+ \int_{\wh X}|f|dV_\eta)
	\end{equation}
	holds uniformly with respect to $\eta= (t, \ep)$, where we have used that both $\varphi$ and $\phi_{\eta}$ are uniformly bounded. Therefore we obtain \eqref{meanvalue} and the point (a) is established.
	
\noindent The idea for (b) is absolutely the same, except that we do not need the factor $\log|s_D|^{2}$, thanks to the additional 
hypothesis. We therefore skip the details.	
\end{proof}
\medskip




\noindent It so happens that the family $(\wh X, \omega_\eta)$ satisfies the following uniform estimates -- this will be very useful later.  

\begin{proposition}
	Let $(E, h_0)$ be a holomorphic Hermitian vector bundle on $\wh X$, where $h_0$ is a smooth metric. Then there exists a constant $C$ such that for any parameter $\eta=(t, \ep)$ whose components are  $\ll 1$ we have 
	\begin{equation}\label{ma20}
		\int_{\wh X}\log\frac{1}{|\sigma_D|^2}|\Lambda_\eta\Theta(E, h_0)|_{h_0} \omega_{\eta} ^n \leq C
	\end{equation}
\end{proposition}

\begin{proof}
	The first observation is that up to a constant, the expression we have to bound is
	\begin{equation}\label{ma21}
		\int_{\wh X}\log\frac{1}{|\sigma_D|^2}\omega_{\wh X}\wedge \omega_{\eta}^{n-1}.
	\end{equation}
	
	Now we estimate \eqref{ma21} by using integration by parts: assume that $\tau$ is smooth function on $\wh X$, such that $\displaystyle \sup_{\wh X}|\tau|= 1$. 
	Then we have
	\begin{equation}\label{ma22}
		\int_{\wh X}\log\frac{1}{|\sigma_D|^2}\ddc\tau\wedge \omega_{\wh X}\wedge \omega_{\eta}^{n-2}= \int_{\wh X}\tau \theta_D\wedge \omega_{\wh X}\wedge \omega_{\eta}^{n-2}- \int_D \tau \omega_{\wh X}\wedge \omega_{\eta}^{n-2}.
	\end{equation}
	Therefore, up to a constant only depending on $D$ and $\{\omega_\eta\}$ this is bounded by 
	\begin{equation}\label{ma23}
		C\int_{\wh X}\omega_{\wh X}^2\wedge \omega_{\eta}^{n-2}\end{equation}	
	where $C$ in \eqref{ma23} is the supreme of the norm of $\theta_D$ with respect to $\omega_{\wh X}$. Applying this to $\tau = \psi_\eta$, we know that \eqref{ma21} is bounded by 
	$$C \int_{\wh X}\omega_{\wh X}^2\wedge \omega_{\eta}^{n-2} . $$
	Iterated applications of this argument shows that \eqref{ma21} is bounded by a constant independent of $\eta$. 
\end{proof}

	
%


\section{Hermite Einstein Metrics in conic setting}\label{coset}

\noindent We show in this section that the results we have established in Section \ref{smoothset} still hold 
if we endow $X$ with the metric $\omega_{\rm orb}$ constructed in Section \ref{orbmetric}.
\smallskip

\noindent The results we have established in Section \ref{MAM} show that for every $\eta= (t, \ep)$ we have a smooth Kähler metric $\omega_\eta$ on $\wh X$ satisfying the following properties.

\begin{itemize}
\smallskip 

\item[(i)]  The metric $\omega_\eta =\pi^*\omega_{\rm orb} +t\omega_{\wh X} +\ddc \psi_\eta$ belongs to the class $\{\pi^*\omega_{\rm orb} +t\omega_{\wh X}\}$.
\smallskip 
\item[(ii)]  We have $\omega_\eta\to \pi^*\omega_{\rm orb}$ as $\eta\rightarrow 0$ and the convergence is 
uniform on compacts in the complement of the support of $D$. Moreover, we have $\phi_\eta\to 0$ in $L^1(\wh X, \pi^*\omega_{\rm orb})$.
\smallskip 

\item[(iii)] The uniform mean value inequality holds. Namely, fix a smooth $(1,1)$-form $\theta$ on $\wh X$. Let $\rho$ be a function on $\wh X$ which is smooth when restricted to the set $\rho \geq -1$ and such that the relation
\begin{equation}\label{eqcon4}
	\Lambda_{\eta} \theta+ \Delta_{\eta}''\rho \geq a  
\end{equation} 
holds, where $a$ is a constant. Here the Laplace and the contradiction are with respect to the metric $\omega_\eta$. Then there exists a positive constant $C> 0$ independent of $\eta$ such that 
	\begin{equation}\label{newcon21}
		\sup_{\wh X}\left(C\log |s_D|^{2}+ \rho\right)\leq C\big(1+ \int_{\wh X}|\rho|dV_\eta\big)
	\end{equation}
where the subscript $\eta$ means that the Laplace operator or norm is with respect to the metric $\omega_\eta$.
\smallskip 

\item[(iv)] In the same set-up as in (iii), assume moreover that we have $\theta\leq C_\theta\pi^\star\omega_X$. Then we obtain a stronger version of \eqref{newcon21}, namely
\begin{equation}\label{newcon21, I}
		\sup_{\wh X}\rho\leq C_2\big(1+ \int_{\wh X}|\rho|dV_\eta\big).
	\end{equation}
\end{itemize}	

\medskip

\noindent Let $\cF$ be a reflexive sheaf on $X$ which is stable with respect to $\omega_{\rm orb}$. As we have already seen,
there exists a desingularisation $\pi: \wh X \rightarrow X$ of $X$ such that $E := \pi^* \cF/\Tor$ is locally free. We denote by $D$ be the exceptional locus of $\pi$, which is assumed to be snc.
As already explained, the bundle $E$ is stable with respect to $\omega_\eta$, and we recall that we have constructed 
a smooth Hermitian metric $h_E$ on $E$ for which the curvature is bounded from below by $\pi^\star\omega_X$. 
\smallskip

\noindent By the theorem of Uhlenbeck-Yau, there exists a metric $h_\eta := h_0\circ H_\eta$ on $E$ which is HE 
with respect to $\omega_\eta$. The endomorphism $H_\eta$ can be decomposed as $$H_\eta =e^{- \frac{1}{r}\rho_\eta} \exp (s_\eta) ,$$ 
where 
\begin{equation}
\Tr (s_\eta)= 0, \qquad \Lambda_\eta\theta_L+ \Delta_\eta(\rho_\eta)= const
\end{equation} 
(notations as in Section 2) together with the normalisation $\int_{\wh X} \rho_\eta dV_\eta=0$. Set $S_\eta = \exp (s_\eta)$.
Thanks to \eqref{newcon21}, \eqref{newcon21, I} we have 
\begin{equation}\label{newcon211}
\max\left\{C\log |s_D|^{2}+ \rho_{\eta}, -\rho_\eta\right\}\leq C\big(1+ \int_{\wh X}|\rho_{\eta}|dV_\eta\big),
\end{equation}
since we have $\theta_L\geq -C\pi^\star\omega_X$ as consequence of \eqref{ross2}.
\smallskip

\noindent We recall that in Section 3 we have established the inequality
\begin{equation}\label{new25, I}
\Delta_{\eta}''\left(\log \Tr(H_\eta)\right)\geq \frac{\Tr\big(\Lambda_\eta \Theta(E, h_{E})\circ H_\eta\big)}{\Tr(H_\eta)}+ C,
\end{equation}
and if we define $\displaystyle \theta:= - \frac{\Tr\big(\Theta(E, h_{E})\circ H_\eta\big)}{\Tr(H_\eta)}$, then we have 
\begin{equation}\label{ross300}
\log \Tr(H_\eta)\leq C\big(1+ \int_{\wh X}|\log \Tr(H_\eta)|dV_\eta\big)
\end{equation}
as consequence of \eqref{newcon21, I}, which applies since $\theta\geq -C\pi^\star\omega_X$.
\medskip

\noindent In this section we are aiming at the following two main results. 

\begin{thm}\label{Estconic, I} There exists a constant $C> 0$ such that the inequalities 
	\begin{equation}\label{estconic1}
		\Tr (H_\eta)\leq C, \qquad -C\leq \rho_\eta\leq -C\log|s_D|^2\qquad \int_{\wh X}|D's_{\eta}|^2\frac{dV_{\eta}}{\log\frac{1}{|s_D|^2}}\leq C
	\end{equation}
	hold, where $s_D$ is the product of sections defining the exceptional divisor of the map $\pi: \wh X\to X$.
\end{thm}

\medskip

\noindent As in the proof of Theorem \ref{HE}, Theorem \ref{Estconic, I} shows that there exists a sub-sequence, $h_\eta$ converging to a metric $h$ of  
$\displaystyle E|_{\wh X \setminus D}$, which verifies the HE equation with respect to $\pi^\star\omega_{\rm orb}$. The fact that the singularities of the metric  
$\omega_{\rm orb}$ are compatible with those of $X$ (at least on an open subset) has the following consequence. 

\begin{thm}\label{reg}
Let $V\subset X$ be an open subset of $X$ such that there exists a finite cover $p: U \rightarrow  V =U/\Gamma$ with the property that $p^* (\omega_{orb})$ is smooth Kähler on $U$. 
Assume moreover that there exists a vector bundle $F$ on $U$ such that $F = p^\star(\cF)^{\star\star}$. Then the metric $p^\star h$ extends to a smooth Hermitian metric on $F$, which moreover satisfies the HE with respect to $p^\star(\omega_{orb})$.
\end{thm}
\medskip

\noindent We first discuss the proof of Theorem \ref{Estconic, I}.

\subsection{Proof of Theorem \ref{Estconic, I}}  After the preparations above, the arguments are very similar to those already discussed in Section \ref{smoothset}. We will therefore go through the proof again, but only provide a detailed arguments when necessary.  
\smallskip

\noindent Set $\displaystyle \Phi(x, y):= \frac{\exp(x-y)- 1}{x-y}$. Since $\omega_\eta$ is smooth and the metric $h_E\circ H_\eta$ is HE, Lemma \ref{W1} shows that the inequality
	\begin{equation}\label{altconic5}
		0 \geq  
		\int_{\wh X}\left\langle \Phi(\log H_\eta)(D' \log H_\eta), D' \log H_\eta\right\rangle dV_\eta+ \int_{\wh X} \Tr \big(\log H_\eta\cdot  \Lambda_{\eta}  \Theta(E, h_0)\big)dV_\eta
	\end{equation}
holds for every $\eta$. 
\medskip

\noindent As $\eta$ tends to zero, two cases can occur:

\begin{enumerate}
	\smallskip
	
	\item[(1)] There exists a constant $C> 0$ such that 
	$$\int_{\wh X}(|\rho_\eta| +  \log {\Tr S_\eta}) dV_\eta \leq C$$
	for all $\eta$ with positive and small enough components.
	\smallskip
	
	\item[(2)] There exist sequences $(\delta_i)_{i\geq 1}$ and  $(\eta_i)_{i\geq 1}$ of parameters
	converging to zero such that 
	$$\int_{\wh X} (|\delta_i \rho_i| + \delta_i \log {\Tr S_i})dV_i = 1$$
We denote by $dV_i$ is the
	volume element corresponding to the metric 
	$\displaystyle \omega_{\eta_i}$.
\end{enumerate}

\subsubsection{The analysis of the case $(1)$} Since $h_E\circ H_\eta$ is HE with respect to $\omega_\eta$, we
 have
\begin{equation}\label{new3conic6}
\max\left\{C\log |s_D|^{2}+ \rho_{\eta}, -\rho_\eta\right\}\leq C, \qquad \Tr (H_\eta)\leq C
\end{equation}
point-wise on $\wh X$, as consequence of \eqref{ross200} and \eqref{newcon211}.

 For this we also use the fact that the trace of $s_\eta$ is equal to zero. Again, the constant 
in \eqref{new3conic6} is independent of $\eta$. 
\medskip

\noindent We combine \eqref{new3conic6} with the inequality \eqref{altconic5},
so we get
\begin{equation}\label{1stcase3conic}
	\int_{\wh X}\left\langle \Phi(\log H_\eta)(D' \log H_\eta), D'\log H_\eta\right\rangle dV_\eta
	\leq
	C\int_{\wh X} \big(1+ \log\frac{1}{|s_D|^2}\big) |\Lambda_{\omega_\eta}\Theta(E, h_0)|dV_\eta
\end{equation}

\noindent By \eqref{ma20}, the RHS of \eqref{1stcase3conic} is uniformly bounded. Then
$$	\int_{\wh X}\left\langle \Phi(\log H_\eta)(D' \log H_\eta), D'\log H_\eta\right\rangle dV_\eta \leq C$$
for some constant $C$ independent of $\eta$.
\noindent Moreover, the estimate \eqref{new3conic6} and the inequality for \eqref{ineqpsi} show that 
\begin{equation}\label{1stcase4conic}
	\left\langle \Phi(\log H_\eta)(D'\log H_\eta), D'\log H_\eta\right\rangle 
	\geq \frac{C}{\log\frac{1}{|s_D|^2}}|D'\log H_\eta|^2
\end{equation}
holds, provided that the positive constant $C\ll 1$ is small enough.
\smallskip

\noindent All in all, we obtain
\begin{equation}\label{1stcase5conic}
	\int_{\wh X}
	\frac{1}{\log\frac{1}{|s_D|^2}}|D'\log H_\eta|^2dV_\eta\leq C, \end{equation}
which completes the proof of Theorem \ref{Estconic, I}.

\subsubsection{The analysis of the second case} We show that this cannot occur. By using the estimates 
\eqref{ross300} and \eqref{newcon211}, we we infer that the inequality
\begin{equation}\label{2ndcaseconic2}
	|u_i|\leq C+ \delta_i\log\frac{1}{|\sigma_D|^2}
\end{equation}
holds true pointwise on $\wh X$.
\medskip

\noindent The following result (analogue of Lemma 5.4 in \cite{Sim88}) was already discussed in Section \ref{smoothset} and exactly the same arguments apply. Therefore, we will not repeat the proof here.
\begin{lemme}\label{weakconic} There exist a subsequence of $(u_i)_{i\geq 1}$ converging weakly to a limit
	$u_\infty$ on compact subsets of $\wh X\setminus D$ such that the following hold.
	\begin{enumerate}
		\smallskip
		
		\item[\rm (1)] The endomorphism $u_\infty$ is non-identically zero and it belongs to the space $H^1$ (i.e. it is in $L^2$ together with its differential).
		\smallskip
		
		\item[\rm (2)] Let $\Psi:\R\times \R\to \R_+$ be a smooth, positive function such that $\displaystyle \Psi(a, b)< \frac{1}{a-b}$ holds for any
		$a> b$. Then we have 
		$$0 \geq  
		\int_{\wh X}\left\langle \Psi(u_\infty)(D'u_\infty), D'u_\infty\right\rangle dV_{\rm orb}+ \int_{\wh X} \Tr \big(u_\infty \Lambda_{0}\Theta(E, h_E)\big)dV_{\rm orb}$$
		where $\Lambda_0$ and $dV_{\rm orb}$ are the contraction with and the volume element corresponding to $\pi^\star\omega_{\rm orb}$, respectively. 
	\end{enumerate}
\end{lemme} 

\medskip

\noindent Next, Lemma \ref{weakconic} combined with the arguments of \cite{Sim88} imply that  the eigenvalues of $u_\infty$ are constant a.e. on $\wh X$. 
We show next that an appropriate eigenspace of $u_\infty$ defines a destabilising subsheaf of $E$. Due to the singularities of $\omega_{\rm orb}$, a few additional arguments are to be invoked in order to make sure that the usual approach goes through. This is done in the rest of the current section.
\smallskip
 
 \noindent Let $\lambda_1,\cdots, \lambda_r$ be the eigenvalues of $u_\infty$. Fix a gamma $\gamma$. Let $p_\gamma :\mathbb R\rightarrow \mathbb R$ be the function such that $p_\gamma (\lambda_i)=1$ for $\lambda_i <\gamma$ and $p_\gamma (\lambda_i)=0$ for $\lambda_i > \gamma$. The projection $\pi_\gamma :=p_\gamma (u_\infty)$ defines a $L^2_1$-subbundle of $E$. 
Now by using Uhlenbeck-Yau, we know that  $\pi_\gamma |_{\wh X\setminus D}$ is holomorphic.  The map $\pi_\gamma$ defines a meromorphic map $\wh X\setminus D \dashrightarrow G (a, r)$, where $G (a, r)$ is the Grassmann manifold associated to $E$, where $a$ is the rank of $\pi_\gamma$.  Then  $\wh X\setminus D \dashrightarrow G (a, r)$ can be extend to be holomorphic morphism outside some subvariety of codimension $2$ in $\wh X$. It defines thus a subsheaf $\wh \cF_\gamma \subset E$ on $\wh X$ such that $E/ \wh \cF_\gamma$ is torsion free.
\smallskip

To finish the proof, it remains to show that there exists a $\gamma$ such that 
\begin{equation}\label{deg} 
\mu_{\omega_{\rm orb}} (\wh\cF_\gamma) \geq \mu_{\omega_{\rm orb}} (E), \qquad \rk \wh\cF_\gamma <\rk E.
\end{equation}
Note that $\pi_\gamma$ is the projection of $E$ to the subsheaf $\wh\cF_\gamma$.
The important remark at this point is that the existence of such $\gamma$ follows from \cite{Sim88}, provided that we are able to justify the formula
\begin{equation}\label{deg1}
c_1 (\wh \cF_\gamma)\wedge \pi^*\omega_{\rm orb} ^{n-1}= \int_{\wh X} \Tr \big(\pi_\gamma \Lambda_{\omega_{\rm orb}} \Theta(E, h_E)\big) \pi^* \omega_{\rm orb} ^n -  \int_{\wh X }  |D'' \pi_\gamma|^2 \pi^* \omega_{\rm orb} ^n.
\end{equation}
This is standard if we replace $\pi^* \omega_{\rm orb}$ with a smooth metric, but we note that the inverse image $\pi^* \omega_{\rm orb}$ has both zeroes and 
poles. We show that it still holds in our case basically because $\omega_{\rm orb}= \omega_X+ \sqrt{-1}\ddbar \varphi$, where $\varphi$ is 
a continuous function on $X$.  
\medskip

\noindent The first thing to remark is that the formula \eqref{deg1} simply asserts that in case of a saturated subsheaf
\begin{equation}\label{deg2}
0\to \wh\cF_\gamma \to E
\end{equation}
one can still computing the degree with respect to the singular metric $\pi^* \omega_{\rm orb}$ by the usual Chern-Weil representative with respect to the induced metric on $\wh\cF_\gamma$ restricted to the set $\wh X\setminus \Sigma$. The "danger" here is that the form given by the Chern-Weil theory is singular, and so is the metric $\pi^\star \omega_{\rm orb}$: in general, their wedge product is not well defined. However, 
in our specific case this is not happening as we next show.
\smallskip

\noindent Let $x_0\in \Sigma\in U$ be a point in the singular set of $\wh\cF_\gamma$ together with an open coordinate subset $U\subset X$. It is a standard fact that there exists a \emph{finite}
set of sections $s^i:= (s^i_1,\dots, s^i_r)_{i=1,\dots, N}$ of $\wh\cF_\gamma|_U$ such that for each point $p$ of $U\setminus \Sigma$ one of the 
$s^i$ above is a frame at the said point. 

\noindent Let $s\in \{s^i\}$ be one of the frames above; the metric induced on $\wh\cF_\gamma$ is given by the coefficients
\begin{equation}\label{deg3}
g_{\beta\ol\alpha}:= \langle \iota(s_\beta), \iota(s_\alpha)\rangle
\end{equation}
so if we write $\displaystyle \iota(s_\alpha)= \sum_{k=1}^m f^k_\alpha e_k$ for some homomorphic functions $f^k_\alpha$ (here $m$ is the 
rank of $E$ and $(e_k)$ is a holomorphic frame), such that the rank of the matrix $\displaystyle (f^k_\alpha)_{k, \alpha}$ is maximal (i.e. $r$) on $U\setminus \Sigma$ locally near the point $p$. It follows that we can write 
\begin{equation}
g_{\beta\ol\alpha}= \sum_{k, l}f^l_\beta\ol f^k_\alpha h_{l\ol k}
\end{equation}
and by choosing $(e_k)$ to be a normal frame, we see that the local weight of the metric induced on the line bundle 
$$L:= \det(\wh\cF_\gamma)$$ 
has singularities of type $\displaystyle \log\sum |g_i|^2$.

\noindent Indeed this can be seen as follows. According to the formula for the Gramm determinant, we have
\begin{equation}\label{deg4}
\det(g_{\beta\ol\alpha})= \big|\iota(s_1)\wedge \dots\wedge \iota(s_r)\big|^2_{\Lambda^rh_E}
\end{equation}
from which we infer that minus of the local weights of the metric for the determinant of $\cF$ (i.e. $dd^c\log\det(g_{\beta\ol\alpha})$) are quasi-psh. 

\noindent In general, let
$T$ be any closed positive current of bidegree $(1,1)$ defined on an open subset $\Omega \subset \CC^n$ and let $\varphi$ be a bounded quasi-psh function on $\Omega$. By the results of Bedford-Taylor, the intersection products
\begin{equation}
T\wedge (dd^c\varphi)^p\wedge (dd^c\Vert z\Vert^2)^q
\end{equation}
are well-defined and moreover they have very good monotonicity properties, cf. e.g. \cite[Chapter 3, Section 3]{Dem}.
Coming back to our case, integrals of type
\begin{equation}\label{deg5}
 \int_{(U, x_0)} dd^c\log\det(g_{\beta\ol\alpha})\wedge (dd^c\varphi)^p\wedge dd^c\Vert z\Vert^2
\end{equation}
are convergent. This already shows that we can compute the degree by using the metric $\pi^* \omega_{\rm orb}$, since we have the equality
\begin{equation}\label{II2}
\int_{\wh X}\theta_L\wedge \pi^\star \omega_{\rm orb}^{n-1}= \lim_{\ep\to 0}\int_{\wh X}\theta_L\wedge \omega_{\ep}^{n-1}
\end{equation}
where $\omega_{\ep}\in \{\pi^\star \omega_{\rm orb}+ \ep\omega_{\wh X}\}$ is smooth and it converges towards $\pi^\star \omega_{\rm orb}$ in such a way that the potentials 
are monotonic, cf. \cite[Thm 13.12]{Dem09} (Here we used the fact that the potentials of $\pi^\star \omega_{\rm orb}$ are continuous. Then we can ask the approximating currents are smooth).
Moreover, we can assume that on the pre-image of the set $\Omega\subset X$ (cf. Section 6) the sequence $(\omega_{\ep})_{\ep> 0}$ is stationary, equal to the inverse image of $\omega_X$ modulo a bounded form  
converging to zero.
\smallskip

\noindent In \eqref{II2} we denote by $\theta_L$ the curvature form of $L= \det(\wh\cF_\gamma)$ induced by the 
metric \eqref{deg3}. Therefore, the total mass of the $(n,n)$-current $\theta_L\wedge \pi^\star \omega_{\rm orb}^{n-1}$
equals the degree of $\wh \cF$ with respect to $\pi^\star \omega_{\rm orb}$.

 \medskip

\noindent Next we recall the formula
\begin{equation}\label{II1}
\Theta_{h_{\wh \cF}}= \pi_\gamma  \Theta(E, h_E)|_{\wh \cF}+ \beta^\star\wedge \beta
\end{equation}
which holds pointwise in the complement of the singularities of $\wh\cF$.
Here we denote by $\beta$ the $(1,0)$--form with values in $\Hom(\wh\cF, E)$
whose adjoint is the second fundamental form $\dbar \pi_\gamma$. 

\noindent By taking the trace we get a $(1,1)$ positive current 
\begin{equation}\label{II3}
-\Tr(\beta^\star\wedge \beta)
\end{equation}
which is in $L^1(\wh X, \wh \omega)$ and we claim that the relation 
\begin{equation}\label{IIl2}
\int_{\wh X}\Tr(\beta^\star\wedge \beta)\wedge \pi^\star \omega_{\rm orb}^{n-1}<\infty
\end{equation}
holds. Indeed, the fact that the (improper) integral on the LHS of \eqref{IIl2} is finite can be seen as consequence of the 
formula \eqref{II1}, since the projection $\pi_\gamma$ is a bounded map. 
\medskip

\noindent It then follows that we have 
\begin{equation}\label{degcep}
\int_{\wh X} \theta_{L}\wedge \pi^\star \omega_{\rm orb}^{n-1}
= \int_{\wh X} \Tr \big(\pi_\gamma  \Theta(E, h_E)\big) \wedge  \pi^\star \omega_{\rm orb}^n -  \int_{\wh X}\Tr(\beta^\star\wedge \beta)\wedge \pi^\star \omega_{\rm orb}^{n-1},	
\end{equation}
and in conclusion
\eqref{deg1} is proved, since we have 
$$\int_{\wh X}\Tr(\beta^\star\wedge \beta)\wedge \pi^\star \omega_{\rm orb}^{n-1}= 
\int_{\wh X }  |D'' \pi_\gamma|^2 \pi^* \omega_{\rm orb} ^n,$$
and moreover, the LHS of \eqref{degcep} equals precisely the degree of $L$ with respect to $\pi^\star \omega_{\rm orb}$, cf. \eqref{II2}.


\medskip

\noindent The rest of the arguments in \cite{Sim88} apply \emph{mutatis mutandis} and the proof of Theorem \ref{Estconic, I} is finished.

\medskip

\subsection{Proof of Theorem \ref{reg}} The basic idea goes back at least to the 
work of Li-Tian, \cite{LT19}. 
Let $\Sigma\subset U$ be the inverse image of the singular loci (of $\cF$ and $X$, it has in any case codimension at least two). We have already seen in Remark \ref{orb} that if we write $\displaystyle \pi^\star H= H_0\exp(\eta)$, then
\begin{equation}\label{rev12}
\sup_{U\setminus \Sigma}|\eta|\leq C< \infty,
\end{equation}
where $H_0$ is a fixed, smooth metric on $F$. 
\smallskip

\noindent Let $\tau$ be any holomorphic section of $F$. We have the formula
\begin{equation}\label{rev13}
  \sqrt{-1}\ddbar |\tau|^2_H= \sqrt{-1}\langle D'_H\tau, D'_H\tau\rangle-
 \sqrt{-1}\langle \Theta(F,H)\cdot \tau,\tau\rangle 
\end{equation}
that we wedge with the $g:= \pi^\star \omega_{\rm orb}$ raised to the power $n-1$.
It follows that we have the equality
\begin{equation}\label{rev14}
  \Delta_g|\tau|^2_H= \Lambda_g\sqrt{-1}\langle D'_H\tau, D'_H\tau\rangle+
 \gamma |\tau|^2_H
\end{equation}
at each point of $U\setminus \Sigma$, where $\gamma$ is a constant.

\noindent It follows that $D'_H\tau$ is in $L^2$ (by multiplying with a truncation function, precisely as we did pages 14-15 in section 6) and since $F|_U$ is trivial, we can construct sufficiently many sections $\tau$ so as to infer that $D'H$ is in $L^2$, and so is $D' _H \tau$. Then the point-wise equation which $H$ satisfies in $U\setminus \Sigma$ extends to $U$ in weak sense, and from this point we are using the elliptic theory. Indeed, the Hermite-Einstein equation is 
\begin{equation}\label{rev15}
\Lambda_g\dbar\big( (D'H)\circ H^{-1}\big)= f
\end{equation}
on $U\setminus \Sigma$. Since we already know that $D' _H \tau$ is $L^2$-integrable, it is sufficient to show that 
\begin{equation}\label{rev16}
\lim_{\ep\to 0}\int_U\sqrt{-1}\partial \Xi_\ep\wedge \dbar \Xi_\ep\wedge \omega_{\rm euc}^{n-1}=0
\end{equation}
This follows from \eqref{conc9}.

\section{Partial results towards Campana-Höring-Peternell conjecture}

Let $X$ be a $3$-dimensional compact Kähler normal space with klt singularities. Throughout this section we denote by $\Sigma \subset X$ the union of the singular loci of $X$ and $\cF$, and let $Z\subset X$ be the locus such that $X\setminus Z$ has locally quotient singularities.  Then $Z$ is either empty, or a finite number of points $(p_i)$ in $X$. We suppose that $\cF$ is stable with respect to a Kähler metric 
$\omega_X$. 
\medskip

\noindent In this section we denote by $h= h_\cF$ the metric obtained in Theorem \ref{HEquot}, which verifies the Hermite-Einstein
equation on $X_0$ and such that it has quotient singularities in the complement of a fixed open nbd $U$ of $Z$. 
Our first goal in the current section is to show that following holds.

\begin{thm}\label{part-11} We assume that the singular set $\Sigma$ of $X$ is the disjoint union of $Z\cup \Sigma_0$, where $\Sigma_0$ is closed and $X$ has quotient singularities at each point of $\Sigma_0$. Then we have 
$$\int_{X\setminus \Sigma}\Delta(\cF, h)\wedge (\omega_{\rm orb}+ \ddc\phi)\geq 0,$$
where the function $\phi$ is such that the support of the form $\omega_{\rm orb}+ \ddc\phi$ is contained in the complement of $U$.  
\end{thm}
\medskip

\begin{remark}
As a consequence of the results in \cite{GK20}, the inequality in Theorem \ref{part-11} states that the orbifold Chern 
classes corresponding to $\cF$ are satisfying the Bogomolov inequality.   
\end{remark} 

\begin{remark}
Actually, it is conjectured in \cite{CHP22} that such a result should hold for any klt space, regardless to its dimension and the structure of its singular points. The arguments proving Theorem \ref{part-11}
equally show that in order to establish the 3-dimensional version of this conjecture it would suffice to
establish a version of the usual regularisation result for $W^{1,1}$-Sobolev spaces. 
\end{remark}
\smallskip

\begin{proof}
\noindent We note that $h$ is smooth in the complement of $\Sigma$, and it is induced by 
an orbifold metric in the complement of an open set $U$ containing the set of points $(p_i)$ (e.g. the union of "coordinate" sets centred at these points). Moreover, according to our hypothesis,
we can assume that 
\begin{equation}\label{bmy-1}
Z\subset U\subset X\setminus \Sigma_0.
\end{equation}
\noindent Moreover, we can assume that 
\begin{equation}\label{bmy-2}
\omega_{\rm orb}|_U= \omega_X,
\end{equation}
this is actually the only part in our arguments where \eqref{bmy-1} is needed.
\smallskip

\noindent Since $h$ is HE with respect to $\omega_{\rm orb}$, the inequality
\begin{equation}\label{pointHE}\Delta (\cF, h) \wedge 
	\omega_{\rm orb} \geq 0
\end{equation}
holds pointwise (as top forms) on $X\setminus \Sigma$. Moreover, since $\omega_{\rm orb}$ is of conic singularity near $\Sigma_0$, thanks to Theorem \ref{reg}, in the neighbourhood of $\Sigma_0$, $\Delta (\cF, h)$ is the push-forward of a smooth $\dbar$-closed $(2,2)$-form from its local uniformisation.  
In particular, $\Delta (\cF, h)$ is well defined and $\dbar$-closed when restricted to the support of the form $\omega_{\rm orb} +\ddc\phi$ (which contains $\Sigma_0$), and moreover satisfies 
\begin{equation}\label{pointHE1}
	\Delta (\cF, h) \wedge \omega_{\rm orb} \geq 0 
\end{equation}
locally near $\Sigma_0$. Here we use the assumption that $Z$ and $\Sigma_0$ are disjoint.
\smallskip

\noindent We also have
\begin{equation} \int_{X\setminus \Sigma}\Delta (\cF, h)\wedge\big(\omega_{\rm orb}+ \sqrt{-1}\ddbar\phi\big) = \int_{X}\Delta (\cF, h)\wedge\big(\omega_{\rm orb}+ \sqrt{-1}\ddbar\phi\big)
\end{equation}
because $\Delta (\cF, h)$ is in $L^1(X, \omega_{\rm orb})$. 

\noindent Without loss of generality we can assumed that the condition $\phi= \O(\Vert z\Vert^2)$ holds at each of the $p_i$ and that 
\begin{equation}\label{bmy9}\sup_{X\setminus Z}|\ddc \phi|_{\omega_{\rm orb}}\leq C< \infty
\end{equation}
since basically the function $\phi$ is a well-chosen potential for $\omega_{\rm orb}$ multiplied with a truncation function.
\medskip

\noindent  Let $(\chi_\ep)_{\ep> 0}$ be the usual family of truncation functions, converging to the characteristic function of 
$X\setminus Z$. A first remark is that we have the equality 
\begin{equation}\label{bmy5}  
\int_{X}\Delta (\cF, h)\wedge\big(\omega_{\rm orb}+ \sqrt{-1}\ddbar\phi\big)= \int_{X}\chi_\ep\Delta (\cF, h)\wedge\big(\omega_{\rm orb}+ \sqrt{-1}\ddbar\phi\big),
\end{equation}
for all $\ep\ll 1$. Moreover, the first term of the RHS of \eqref{bmy5} 
\begin{equation}
\int_{X}\chi_\ep\Delta (\cF, h)\wedge\omega_{\rm orb} \geq 0
\end{equation}
is positive, as consequence of \eqref{pointHE} and \eqref{pointHE1}.

\smallskip

\noindent The next identity we have to establish is
\begin{equation}\label{ymb1}
\int_{X} \chi_\ep \Delta (\cF,h) \wedge 
	\ddbar \phi= -\int_{X}\partial \phi\wedge \dbar\chi_\ep \wedge \Delta (\cF, h)
\end{equation}
for every positive $\ep$. This of course looks like an obvious consequence of Stokes theorem, but it is not so immediate since we are integrating over a non-compact domain singular forms.
In our specific situation, it holds thanks to \eqref{bmy-2} combined to the point (2) of Theorem \ref{Chern1}.
\smallskip

\noindent As in the proof of Theorem \ref{Chern1}, it is easy to see that
\begin{equation}\label{bmy4}  
	\sup_{X_{\rm reg}}|\partial \phi\wedge\dbar \chi_\ep|_{\omega_{\rm orb}}\leq C
\end{equation}
by a direct calculation, and on the other hand we claim that there exists a constant $C> 0$ such that 
\begin{equation}\label{bmy8} 
\Delta (\cF,h) \wedge \omega_{\rm orb} \leq C |\Theta(\cF, h)|^2dV_{\omega_{\rm orb}}
\end{equation}
at each point of $X_{\rm reg}$ -- this has already been verified in Section 5.  
\medskip

\noindent In conclusion, we have proved that 
\begin{equation}\label{ymb-3}
\lim_{\ep\to 0}\int_{X} \chi_\ep \Delta (\cF, h) \wedge 
	\ddbar \phi= 0
\end{equation}
so Theorem \ref{part-11} is established.
\end{proof}
\medskip

\begin{remark}
We assume that, instead of choosing $\phi$ such that $\omega_{\rm orb}+ \sqrt{-1}\ddbar\phi$ has compact support we impose the 
following condition: \emph{for any point $x\in X_{\rm sing}$ contained in the singular loci of $X$ we have}
$$\frac{\phi}{\sum |f_j|^2}\in \CC^\infty(X, x)$$
where $(f_j)$ is the ideal corresponding to $X_{\rm sing}$. Then the arguments used for Theorem \ref{part-11} show that 
$$\int_{X\setminus \Sigma}\Delta(\cF, h_\cF)\wedge (\omega_{\rm orb}+ \ddc\phi)\geq 0$$ holds true. 
\end{remark}

\medskip

\noindent We end this section with the proof of Theorem \ref{Chern2}.
\smallskip

\begin{proof}
It may be that Theorem \ref{Chern2} is an important step towards the general case. Indeed, it is expected that any klt K\"ahler space $X$ admits a so-called \emph{partial resolution of singularities},
i.e. there exists a birational map 
\begin{equation}
\pi: \wt X\to X
\end{equation}
such that $\wt X$ has only quotient singularities and $\pi$ is an isomorphism over the complement of the non-orbifold loci of $X$
(i.e. over the complement of a finite number of points). 
\smallskip

\noindent In the context of Theorem \ref{Chern2} this is a hypothesis, and consider the inverse image sheaf 
\[\wt \cF:= \big(\pi^\star(\cF)/ {\rm Tor}\big)^{\star\star}.\] 
\smallskip

\noindent It has two important properties that we next discuss, namely:
\begin{itemize}
\smallskip 

\item It is isomorphic with the $\pi$-inverse image of $\cF$ over the orbifold subset of $X$.
\smallskip 

\item $\wt \cF$ admits a $\mathbb Q$-sheaf structure in the complement of a finite number of points $(p_i)_{i=1,\dots, N}\subset \wh X$ of $\wh X$.
\end{itemize}
To see this, we first consider $\wh \cF:= \pi^\star (\cF)/\Tor $, and let $0$ be a point in the support of the exceptional divisor of $\pi$, together with the 
local uniformization \[\tau: U\to V= U/\Gamma.\]
The double dual of the inverse image $\tau^\star \wh \cF$ is a reflexive sheaf $\mathcal F_U$ on a open subset of $\mathbb C^3$, hence it is a vector bundle 
in the complement of a finite set of points. We then take the $\Gamma$-invariant direct image of $\mathcal F_U$,
which coincides with  $\wh \cF$ outside a set of codimension at least two. Therefore, the bi-dual of $\wh \cF$ is isomorphic with
the $\Gamma$-invariant direct image of $\mathcal F_U$, which shows that the bullets above hold true.
\smallskip

\noindent Let $\omega_{\wh X}$ be a metric with orbifold singularities on $\wh X$. Then for each positive $\ep> 0$ the form
\begin{equation}\label{coho6}
\omega_\ep:= \pi^\star\omega_X+ \ep\omega_{\wh X}
\end{equation}
 is a metric with quotient singularities.
  
Next we see that $\wh \cF$ is stable with respect to $\omega_\ep$, as soon as the parameter $\ep\ll 1$ is small enough. 
Let $h_\ep$ be the corresponding HE metric constructed in Theorem \ref {HEquot}. Then we have the inequality
\begin{equation}\label{coho7}
\int_{\wh X_0}\Delta (\wh \cF, h_\ep)\wedge\omega_{\ep}\geq 0
\end{equation}
 which follows by the usual differential geometry calculation, combined with the estimates in Theorem \ref {HEquot}
 in order to justify the convergence of the integral in \eqref{coho7}. The main point is to show that the equality
\begin{equation}\label{coho5}
\int_{\wh X}\Delta (\wh \cF, h_\ep)\wedge(\omega_{\ep}+ \sqrt{-1}\ddbar\phi)= \int_{\wh X}\Delta (\wh \cF, h_\ep)\wedge \omega_{\ep}
\end{equation}
holds for any function $\phi$ which is smooth in orbifold sense. Modulo the use of local uniformizations, the argument to establish 
\eqref{coho5} is the same as the one we used in Theorem \ref{Chern1}, so we will not reproduce it here. 
\smallskip

\noindent In particular, \eqref{coho5} holds for a function $\phi$ such that $\omega_{\ep}+ \sqrt{-1}\ddbar\phi$ is equal to zero
locally near the points $p_i$. Then we have the equality
\begin{equation}\label{coho9}
\int_{\wh X}\Delta (\wh \cF, h_\ep)\wedge\big(\omega_{\ep}+ \sqrt{-1}\ddbar\phi\big)= 
\int_{\wh X}\Delta (\wh \cF, h_1)\wedge\big(\omega_{\ep}+ \sqrt{-1}\ddbar\phi\big)
\end{equation}
justified by the fact that the form $\omega_{\ep}+ \sqrt{-1}\ddbar\phi$ has compact support contained in the 
open subset of $\wh X$ where $\wh \cF$ is a $\mathbb Q$--sheaf. 

\noindent Next we write 
\[\omega_{\ep}+ \sqrt{-1}\ddbar\phi= \pi^\star(\omega_X+ \sqrt{-1}\ddbar\psi)+ \ep\eta,\]
where $\psi$ is chosen in such a way that $\omega_X+ \sqrt{-1}\ddbar\psi$ has support in a 
compact set $K\subset X$ contained in the orbifold loci of $X$. Moreover, we choose $\eta$ with compact support in 
a set $\wh K$ contained in the 
complement of $(p_i)_i$'s.  All in all, we obtain the inequality
\begin{equation}\label{coho10} 
\int_{\wh X}\Delta (\wh \cF, h_1)\wedge\pi^\star\big(\omega_{X}+ \sqrt{-1}\ddbar\psi\big)\geq -\ep 
\int_{\wh X}\Delta (\wh \cF, h_1)\wedge \eta
\end{equation}
and as a last step we are taking the limit in \eqref{coho10} as $\ep\to 0$.
\end{proof}

\end{document}